\documentclass[11pt,reqno]{amsart}
\usepackage{graphicx,amsmath,amsthm,verbatim,amssymb,enumerate}
\usepackage{xcolor}
\definecolor{darkgreen}{rgb}{0,0.75,0}
\definecolor{darkred}{rgb}{0.75,0,0}
\definecolor{darkmagenta}{rgb}{0.5,0,0.5}
\usepackage[colorlinks,citecolor=darkgreen,linkcolor=darkred,urlcolor=darkmagenta]{hyperref}
\newcommand{\zb}[1]{{\tt \href{https://zbmath.org/?q=an:#1}{Zbl#1}}}
\newcommand{\mr}[1]{{\tt \href{http://www.ams.org/mathscinet-getitem?mr=#1}{MR#1}}}
\newcommand{\arxiv}[1]{{\tt \href{http://arxiv.org/abs/#1}{arXiv:#1}}}

\newcommand{\old}[1]{}

\newcommand{\abs}[1]{{\left\vert\kern-0.25ex #1
    \kern-0.25ex\right\vert}}

\DeclareRobustCommand{\SkipTocEntry}[5]{}

\newtheorem{theorem}{Theorem}[section]
\newtheorem{prop}[theorem]{Proposition}

\newtheorem{lemma}[theorem]{Lemma}
\newtheorem{lem}[theorem]{Lemma}

\theoremstyle{remark}

\newtheorem{rem}{Remark}
\newtheorem*{notn}{Notation}

\numberwithin{counter}{section}

\theoremstyle{definition}
\newtheorem{definition}[theorem]{Definition}

\def\ind{\mathbf{1} } 

\def\00{\mathbf{0}}

\def\F{\mathcal{F}}
\def\E{\mathcal{E}}
\def\T{\mathcal{T}}

\def\N{\mathbb{N}}
\def\Z{\mathbb{Z}}

\def\R{\mathbb{R}}
\def\EE{\mathbb{E}}
\def\PP{\mathbb{P}}

\newcommand\norm[1]{\left\lVert#1\right\rVert} 

\begin{document}

\begin{abstract}
Let $(M,d,\mu)$ be a uniformly discrete metric measure space satisfying space homogeneous volume doubling condition. We consider discrete time Markov chains on $M$ symmetric with respect to $\mu$ and whose one-step transition density is
comparable to $ (V_h(d(x,y)) \phi(d(x,y))^{-1}$, where $\phi$ is a positive continuous regularly varying function with index $\beta \in (0,2)$ and $V_h$ is the homogeneous volume growth function.
Extending several existing work by other authors, we prove global upper and lower  bounds for $n$-step transition probability density that are sharp up to constants.
\end{abstract}

\title{ Transition probability estimates for long range random walks}
\author{Mathav Murugan, Laurent Saloff-Coste$^\dagger$}
\thanks{$\dagger$Both the authors were partially supported by NSF grants  DMS 1004771 and DMS 1404435.}
\address{Mathav Murugan, Department of Mathematics, University of British Columbia and Pacific Institute for the Mathematical Sciences, Vancouver, BC V6T 1Z2, Canada. mathav@math.ubc.ca}
\address{Laurent Saloff-Coste,  Department of Mathematics, Cornell University, Ithaca, NY 14853, USA.  lsc@math.cornell.edu}
\date{\today}
\maketitle
\section{Introduction}
Let $(M,d,\mu)$ be a countable, metric measure space. We assume that $(M,d,\mu)$ is \emph{uniformly discrete}, that is there exists $a>0$ such that any two distinct points $x,y \in M$ satisfy $d(x,y) >a$.
The main example we have in mind are connected graphs with its natural graph metric.

Further we assume that the measure $\mu$ is comparable to the counting measure in the following sense: there exists $C_\mu \in [1,\infty)$ such that $\mu_x = \mu\left( \{x\}\right)$ satisfies
\begin{equation} \label{e-count}
 C_\mu^{-1} \le \mu_x \le C_\mu
\end{equation}
for all $x \in M$.
Let $B(x,r):= \{ y \in M : d(x,y) \le r \}$ be the  ball in $M$ for metric $d$ with center $x$ and radius $r \ge 0$.
Let $V(x,r):= \mu(B(x,r))$ denote the volume of the  ball centered at $x$ of radius $r$.

We consider metric measure spaces $(M,d,\mu)$ satisfying the following uniform volume doubling assumption:
there exists a non-decreasing function $V_h :[0,\infty) \to (0,\infty)$ and constants $C_D,C_h  \ge 1$ such that
\begin{equation} \label{e-vd}
 V_h(2r) \le C_D V_h(r)
\end{equation}
for all $r >0$ and
\begin{equation} \label{e-hom}
C_h^{-1} V_h(r) \le V(x,r) \le C_h V_h(r)
\end{equation}
for all $x \in M$ and for all $r >0$. It can be easily seen from \eqref{e-vd} that
\begin{equation} \label{e-vc}
 \frac{V_h(R)}{V_h(r) } \le C_D \left( \frac{R}{r} \right)^\alpha
\end{equation}
for all  $0 < r \le R$ and for all $\alpha \ge \log_2 C_D$. For the rest of the work, we assume that
our metric measure space $(M,d,\mu)$ is uniformly discrete satisfying \eqref{e-count}, \eqref{e-vd} and \eqref{e-hom}.

In this paper, we consider discrete time Markov chains $\{ X_n , n \ge 0, \PP^x, x \in M\}$ that are reversible with respect to the measure $\mu$.
That is the  transition probabilities $p(x,y)$ satisfy
\begin{equation} \label{e-reverse}
 p(x,y) \mu_x = p(y,x) \mu_y
\end{equation}
for all $x,y \in M$.
The associated Markov operator $P$, given by
\[
 P f(x) = \sum_{x \in M} p(x,y) f(y)
\]
is self-adjoint in $\ell^2(M,\mu)$. We assume that the walk has infinite lifetime, that is $\sum_{z \in M} p(x,z) =1$ for all $x \in M$.

For $n \in \N := \{0,1,\ldots\} $, let $p_n$ denote the $n^{\operatorname{th}}$ iterated power of $p$, that is
\[
p_0(x,y) = \delta_{x,y}:= \begin{cases} 0, &\mbox{if } x \neq y,\\
1, & \mbox{if }  x=y, \end{cases}
\]
and
\[
 p_n(x,y) = \sum_{z \in M} p_{n-1}(x,z) p (z,y) ,\hspace{.5cm} n \ge 1.																																
\]
In other words, $p_n(x,y)$ is the transition function of the random walk $X_n$, i.e.,
\[
 p_n(x,y) = \PP^x(X_n=y),
\]
or the kernel of the operator $P^n$ with respect to counting measure. Define the heat kernel, that is, the kernel of $P^n$ with respect to $\mu$, or the transition density of $X_n$, by
\[
 h_n(x,y):= \frac{p_n(x,y)}{\mu_y}.
\]
Clearly $h_n$ is symmetric, that is, $h_n(x,y)=h_n(y,x)$. As a consequence of the semigroup law $P^{m+n}=P^mP^n$, the heat kernel satisfies the Chapman-Kolmogorov equation
\begin{equation}\label{e-ck}
 h_{n+m}(x,y)= \sum_{z \in M} h_n(x,z) h_m(z,y) \mu_z
\end{equation}
for all $x,y \in M$ and for all $n,m \in \N$. Define the jump kernel (or conductance) $J :=h_1$ as the kernel of $P$ with respect to $\mu$.

We consider random walks with unbounded range and the following conditions may be imposed on the jump kernel $J$.
We say that $J$ satisfies \ref{ujp} , if there exists $C>0$ such that
\begin{equation} \label{ujp}
 J(x,y) \le \frac{C}{(1+d(x,y))^\beta V_h(d(x,y))}       \tag*{($UJP(\beta)$)}
\end{equation}
for all $x,y \in M$. Similarly, we say  $J$ satisfies \ref{ljp} , if there exists $c>0$ such that
\begin{equation} \label{ljp}
 J(x,y) \ge \frac{c}{(1+d(x,y))^\beta V_h(d(x,y))}       \tag*{($LJP(\beta)$)}
\end{equation}
for all $x,y \in M$. If $J$ satisfies both \ref{ujp} and \ref{ljp}, we say $J$ satisfies $JP(\beta)$.

We wish to prove the following estimates for the heat kernel $h_n$. We say $h_n$ satisfies \ref{uhkp}, if there exists $C>0$ such that
\begin{equation} \label{uhkp}
 h_n(x,y) \le C \left( \frac{1}{ V_h(n^{1/\beta})} \wedge \frac{n}{ (1 + d(x,y))^\beta V_h(d(x,y))} \right) \tag*{($UHKP(\beta)$)}
\end{equation}
for all $n \in \N^*$ and for all $x ,y \in M$. Similarly, we say $h_n$ satisfies \ref{lhkp}, if there exists $c>0$ such that
\begin{equation}\label{lhkp}
 h_n(x,y) \ge c  \left( \frac{1}{ V_h(n^{1/\beta})} \wedge \frac{n}{ (1 + d(x,y))^\beta V_h(d(x,y))} \right) \tag*{($LHKP(\beta)$)}
\end{equation}
for all $n \in \N^*$ and for all $x ,y \in M$. If $h_n$ satisfies both \ref{uhkp} and \ref{lhkp}, we say $h_n$ satisfies $HKP(\beta)$.

\begin{rem}
\begin{enumerate}[(a)]
 \item By \eqref{e-count}, we may equivalently replace $h_n$ by $p_n$ in \ref{uhkp} and \ref{lhkp}.
 \item One of the advantages of working in the setting on uniformly discrete metric spaces (as opposed to connected graphs) is that  $JP(\beta)$
  and $HKP(\beta)$ can be easily generalized if we replace $(1+d(x,y))^\beta$ by a regularly varying function of index $\beta$. This remark will be made precise in last section (see Theorem \ref{t-main}).
\end{enumerate}
\end{rem}
Let $\mathcal{E}$ denote the Dirichlet form associated with $P$ defined by
\[
 \mathcal{E}(f,f):=  \langle (I-P) f,f \rangle = \frac{1}{2} \sum_{x,y \in M} (f(x)-f(y))^2 J(x,y) \mu_x \mu_y
\]
for all $x,y \in \ell^2(M,\mu)$, where $\langle \cdot , \cdot \rangle $ denotes the inner product in $\ell^2 (M,\mu)$. We abbreviate $\mathcal{E}(f,f)$ by $\mathcal{E}(f)$.
Since $\mathcal{E}$ is a Dirichlet form, we have
\begin{equation}\label{e-cont}
 \mathcal{E}( (f-t)^+ \wedge s ) \le \mathcal{E}(f)
\end{equation}
for all $s,t \in [0,\infty)$ and for all $f \in \ell^2(M,\mu)$. We will frequently work with the corresponding continuous time Markov chain
defined by $Y_t := X_{N(t)}$ where $N(t)$ is a standard Poisson process independent of $(X_n)_{n \in \N}$. We denote the transition probability density of $Y_t$ with respect to $\mu$ by $q_t$, that is
\begin{equation} \label{e-qdef}
 q_t(x,y) := \frac{\PP^x(Y_t=y)}{ \mu_x} = \sum_{k=0}^\infty \frac{e^{-t}t^k}{k!}  h_k(x,y).
\end{equation}
By $\norm{f}_p$ we denote the $p$-norm in $\ell^p(M,\mu)$, where $1 \le p \le \infty$.
The main result of this paper is the following.
\begin{theorem} \label{t-main}
 Let $(M,d,\mu)$ be a countable, uniformly discrete, metric measure space satisfying \eqref{e-count}, \eqref{e-vd} and \eqref{e-hom}.
 Assume $\beta \in (0,2)$ and $\phi:[0,\infty) \to (0,\infty)$ be a continuous, positive regularly varying function with index $\beta$ such that
 $\phi(x) = \left( (1+x)l(x) \right)^\beta $ where $l$ is slowly varying function.
 Let $\mathcal{E}$ be a Dirichlet form
and $C_1 >0$ be a constant such that the jump kernel $J=h_1$ with respect to $\mu$  satisfies
 \begin{equation}\label{e-mn}
 C_1^{-1}\frac{1}{V_h(d(x,y)) \phi(d(x,y))} \le  J(x,y)=J(y,x) \le C_1 \frac{1}{V_h(d(x,y)) \phi(d(x,y))}
 \end{equation}
for all $x,y \in M$. Then there exists $C_2>0$ such that
\begin{align*}
 h_n(x,y) &\le C_2 \left( \frac{1}{ V_h(n^{1/\beta}l_{\#}(n^{1/\beta}))} \wedge \frac{n}{ V_h(d(x,y))\phi(d(x,y))} \right) \\
 h_n(x,y) &\ge C_2^{-1} \left( \frac{1}{ V_h(n^{1/\beta} l_{\#}(n^{1/\beta}))} \wedge \frac{n}{  V_h(d(x,y)) \phi(d(x,y))} \right)
\end{align*}
for all $n \in \N^*$ and for all $x,y \in M$, where $l_{\#}$ denotes the de Bruijn conjugate of $l$.
\end{theorem}
\begin{rem} \label{r-dis}
 Similar estimates can be easily obtained for the continuous time kernel $q_t$ using \eqref{e-qdef} and the above Theorem. However, in general it is not easy to obtain estimates on $h_n$ given estimates on $q_t$.
\end{rem}

Such estimates were first obtained in \cite{BL} for discrete time Markov chain on $\Z^d$.
Other early works include \cite{CK1}, \cite{CK2} which  concerns jump process on
metric measure spaces with homogeneous volume growth that are subsets of metric spaces having a scaling structure (see (1.15) in \cite{CK2}).
We do not require any such scaling structure, however we require that our metric space is uniformly discrete.
The relationship between heat kernel upper bounds for jump processes and exit time estimates is explored in \cite{BGK}
and the relationship between parabolic Harnack inequality and heat kernel estimates for jump processes is studied in \cite{BBK}.
All these works with the exception of \cite{BL} are for continuous time jump processes.

In light of Remark \ref{r-dis}, we find it advantageous to work in discrete time setting.
It seems appropriate to have a detailed self-contained proof of Theorem \ref{t-main}.
It is a technically interesting open problem to generalize Theorem \ref{t-main} if we replace  homogeneous volume doubling assumption given by \eqref{e-hom} and \eqref{e-vd} by the more general volume doubling condition:
There exists $C_D >1$ such that $V(x,2r) \le C_D V(x,r)$ for all $x \in M$ and for all $r>0$.

It may be useful to comment on the detailed assumption \eqref{e-mn}
in Theorem 1.1 or, for simplicity, on \ref{ujp} and \ref{ljp}.
Roughly speaking, we require that the jump kernel $J$ has a uniform power
type decay with parameter $\beta\in (0,2)$ and it is natural to ask if
this assumption can be weakened. The answer to this question depends greatly on whether one insists on obtaining  matching two-sided bounds in time-space as we do here or if one is content with sharp information on the ``on diagonal behavior" of the iterated kernel $h_n$.
The answer also depend on how much one is willing to assume on the
underlying metric space.

In the context considered here where \eqref{e-count} and \eqref{e-vd} are the main assumptions
on the underlying metric space and if one insists
on obtaining  matching two-sided bounds in time-space,
it seems very difficult, both technically and conceptually,
to relax the assumption on $J$. See the related results in \cite{GHL}.

To help the reader gain some insight on the difficulties involved,
we consider several options and point to some related works.
\begin{enumerate}[(A)]
\item What happens if $\beta\ge 2$?   Even in the simplest setting of
$\mathbb Z$ or $\mathbb R$, no sharp two-sided time-space estimates
are available for the iterated kernel $h_n$ when $\beta \ge 2$ (especially, when $\beta=2$!).  In general,
in order to describe the ``on-diagonal" behavior of $h_n$,
very restrictive additional hypotheses on the underlying metric measure space are necessary. See \cite{BCG}, \cite{SZ1} and \cite{MS}.

\item What happens if $J(x,y)\simeq 1/(V_h(d(x,y))\phi((d(x,y)))$ with $\phi$
growing slower than a power function, e.g.,
$\phi (t)=(1+\log (1+t))^\gamma$, $\gamma>1$? Under some additional hypotheses
on the underlying space, it is possible to study the ``on-diagonal" behavior
of the iterated kernel $h_n$. In many cases, the ``on-diagonal" decay of $h_n$
is expected (or known) to be faster than any inverse power function.
In such cases, sharp two-sided time-space estimates are extremely difficult and not really expected (even the form such estimates should take is unknown). See \cite{SZ2} and \cite{SZ3}.

\item  What happens if $J$ oscillates? For instance $J$ could be
radial with a lacunary power like decay including long intervals on
which $J=0$. In such cases, as in the case when $\beta\ge 2$, the
on-diagonal behavior of $h_n$ will typically depend on making additional hypotheses on  the underlying space and sharp two-sided time-space estimates may be very difficult to obtain.
\end{enumerate}

The paper is organized as follows. The rest of the paper is devoted to the proof of Theorem \ref{t-main}. We first prove Theorem \ref{t-main} in a slightly restricted setting under the assumption that $J$ satisfies
\ref{ujp} and \ref{ljp}. In Section \ref{s-gen}, we use a change of metric argument to handle the general case. In Section \ref{s-ond}, we use a Nash inequality to obtain on-diagonal upper bounds, \emph{i.e.\ }upper bounds on $h_n(x,x)$ as $x \in M$ and  $n \in \N^*$ varies. In Section \ref{s-ub}, following \cite{BGK} we use  Meyer decomposition of the jump process along with Davies perturbation method to obtain upper bounds on the transition probability density. In Section \ref{s-phi}, we prove a parabolic Harnack inequality using an iteration argument due to Bass and Levin \cite{BL}. In
Section \ref{s-hke}, we use the parabolic Harnack inequality to obtain two-sided bounds on the transition probability density.

\section{On-diagonal upper bound}\label{s-ond}
In this section, we prove a Nash inequality using  `slicing techniques' developed in  \cite{BCLS}. This approach of proving Sobolev-type inequalities is outlined in Section 9 of \cite{BCLS}.
We remark that different Nash inequalities developed in \cite{CK1} and \cite{CK2} would yield the desired on-diagonal upper bounds as well.

We say $\mathcal{E}$ satisfies Nash inequality \ref{N}, if there exist constants $\alpha, C_1,C_2 \in (0,\infty)$ such that
\begin{equation} \label{N}
\norm{f}_2 \le  C_1 \left(\left(\frac{R^{\alpha}}{V_h(R)}\right)^{\frac{\beta}{\alpha}} \left( \mathcal{E}(f) + C_2 R^{-\beta} \norm{f}_2^2 \right) \right)^{\frac{\alpha}{2(\alpha+\beta)}} \norm{f}_1^ {\frac{\beta}{\alpha+\beta}} \tag*{ ($N(\beta)$)}
\end{equation}
for all $R >0$ and for all $f \in \ell^1(M,\mu)$.
We obtain, on-diagonal upper bound on $q_t(x,x)$ and $h_n(x,x)$ as a consequence of Nash inequality \ref{N}.
Before proving Nash inequality, we show that Nash inequality \ref{N} implies the desired on-diagonal estimate on $q_t$.

\begin{prop} \label{p-ultra}
If the Dirichlet form $\mathcal{E}$ satisfies \ref{N}, then there exists constant $C_4>0$ such that
\begin{equation}
 \label{e-c-dub} q_t(x,x) \le C_4 \frac{1}{V_h(t^{1/\beta})}
\end{equation}
for all $t >0$ and  for all $x \in M$.
\end{prop}

\begin{proof}
Let $C_1$ and $C_2$ be the constants from \ref{N}. Define the semigroup $T_t^R$ and an operator $A_R$ by
\begin{align*}
 T_t^R &= e^{-C_2 R^{-\beta }t} e^{-t(I-P)} \\
A_R &= (1+C_2 R^{-\beta}) I - P.
\end{align*}
It is easy to check that $-A_R$ is the infinitesimal generator of the semigroup $T_t^R$ and that $T_t^R$ is equicontinuous contraction on $\ell^1(M,\mu)$ and $\ell^\infty (M,\mu)$ with
\[
 \sup_{t} \norm{T_t^R}_{1 \to 1} , \sup_t \norm{T_t^R}_{\infty \to \infty} \le 1.
\]
By \ref{N}, we have
\[
 \theta_R(\norm{f}_2^2) \le \langle A_R f,f \rangle, \hspace{0.5cm} \forall f \in \ell^2(M,\mu), \norm{f}_1 \le 1
\]
where
\[
 \theta_R(t) = C_1^{-1} \left( \frac{V_h(R)}{R^\alpha} \right)^{\beta/\alpha} t^{1+\frac{\beta}{\alpha}}.
\]
Hence by Proposition II.1 of \cite{Cou}, there exists $C_3>0$ such that
\[
 e^{-C_2 R^{-\beta}t} \sup_{x} q_t(x,x) = \norm{T_t^R}_{1 \to \infty} \le C_3 t^{-\alpha/\beta} \left( \frac{R^\alpha}{V_h(R)}\right)
\]
for all $t,R >0$. Fixing $R= t^{1/\beta}$, we get
\[
\sup_{x} q_t(x,x) \le \frac{C_3 e^{C_2} }{V_h(t^{1/\beta})}
\]
which proves \eqref{e-c-dub}.
\end{proof}

Define
\[
 \F :=  \{ f\in \ell^1(M,\mu) : f \ge 0 \}
\]
to be the class of non-negative $\ell^1$ functions. It is easy to check that $\F$ satisfies the following properties:
\begin{enumerate}[(a)]
 \item (Stability under slicing) $f \in \F$ implies $(f-t)^+ \wedge s \in \F$ for all $s ,t \ge 0$.
 \item $\F$ is a cone, that is for any $t>0$ and $f \in \F$, we have $t f \in \F$.
 \item $\F \subset \ell^p(M,\mu)$ for all $p \in [1,\infty]$.
\end{enumerate}

Let $W(f)$ be a semi-norm on $\F$. We recall some properties introduced in \cite{BCLS}. We say $W$ satisfies
\ref{h-inf} if there exists a constant $A_\infty^+$ such that
\begin{equation}
\label{h-inf} \tag*{ ($H_\infty^+$)} W \left( (f-t)^+ \wedge s \right) \le A_\infty^+ W(f).
\end{equation}
 for all $f \in \F$ and for all $s,t \ge 0$.
For any $\rho > 1$, $k \in \Z$ and any function $f \in \F$, set
\[
 f_{\rho,k} = (f - \rho^k)^+ \wedge \rho^k ( \rho - 1)
\]
which is also in $\F$.
Fix $ l>0$ and $\rho >1$. We say that $W$ satisfies the condition \ref{h-lr} if there exists a constant $A_l(\rho)$ such that

\begin{equation}
 \label{h-lr} \tag*{ ($H^\rho_l$)} \left( \sum_{k \in \Z} W(f_{\rho,k})^l \right)^{1/l} \le A_l(\rho) W(f).
\end{equation}
for all $ f \in \F$.
The properties \ref{h-inf} and \ref{h-lr} are preserved under positive linear combinations of semi-norms as shown below.
\begin{lem} \label{l-H1}
Let $N_1$ and $N_2$ be semi-norms on $\F$ satisfying  \ref{h-inf} with constants $A_{\infty,1}$, $A_{\infty,2}$ such that for all
$f \in \mathcal{F}$ and for all $s,t \ge 0$
\begin{align*}
N_1( (f-t)^+ \wedge s ) &\le  A_{\infty,1} N_1(f) \\
N_2( (f-t)^+ \wedge s ) &\le  A_{\infty,2} N_2(f).
\end{align*}
Then for any $c_1,c_2 \ge 0$, the semi-norm $N= c_1 N_1 + c_2 N_2$  satisfies  $(H^+_\infty)$ with
\begin{equation*}
N( (f-t)^+ \wedge s ) \le  \max(A_{\infty,1},A_{\infty,2}) N(f)
\end{equation*}
for all $f \in \mathcal{F}$ and $s,t \ge 0$.
\end{lem}
\begin{proof}
\begin{align*}
N( (f-t)^+ \wedge s ) &= c_1 N_1 ( (f-t)^+ \wedge s ) + c_2 N_2 ( (f-t)^+ \wedge s ) \\
& \le c_1 A_{\infty,1} N_1 (f) + c_2 A_{\infty,2} N_2(f) \\
& \le  \max (A_{\infty,1},A_{\infty,2}) N(f).
\end{align*}

\end{proof}

\begin{lem} \label{l-H2}
Fix  $\rho > 1$ and $l > 0$.
Let $N_1$ and $N_2$ be semi-norms on $\F$ satisfying \ref{h-lr} with constants $A_{l,1}(\rho)$, $A_{l,2}(\rho)$ such that for all
$f \in \mathcal{F}$ and for all $s,t \ge 0$
\begin{align*}
 \left( \sum_{k \in \mathbb{Z}} N_1(f_{\rho,k})^l \right)^{1/l} &\le  A_{l,1}(\rho) N_1(f) \\
 \left( \sum_{k \in \mathbb{Z}} N_2 (f_{\rho,k})^l \right)^{1/l} &\le  A_{l,2}(\rho) N_2(f).
\end{align*}
Then for any $c_1,c_2 \ge 0$, the semi-norm $N= c_1 N_1 + c_2 N_2$  satisfies  \ref{h-lr}  with
\begin{equation*}
\left( \sum_{k \in \mathbb{Z}} N (f_{\rho,k})^l \right)^{1/l} \le  2^{(l+1)/l} \max(A_{l,1}(\rho),A_{l,2}(\rho)) N(f)
\end{equation*}
for all $f \in \mathcal{F}$ and $s,t \ge 0$.
\end{lem}

\begin{proof}
\begin{align*}
 \left( \sum_{k \in \mathbb{Z}} N (f_{\rho,k})^l \right)^{1/l}  & \le  \left( \sum_{k \in \mathbb{Z}} ( c_1 N_1 (f_{\rho,k}) + c_2 N_2(f_{\rho,k}) )^l \right)^{1/l}  \\
& \le  2 \left( \sum_{k \in \mathbb{Z}} \left(  c_1^l N_1 (f_{\rho,k})^l + c_2^l N_2(f_{\rho,k})^l \right) \right)^{1/l}  \\
& \le  2  \left( A_{l,1}(\rho)^l c_1^l N_1 (f)^l + A_{l,2}(\rho)^l c_2^l N_2(f)^l \right)^{1/l}  \\
& \le    2^{1/l} 2 \left(  A_{l,1}(\rho) c_1 N_1 (f) + A_{l,2}(\rho) c_2 N_2(f) \right)  \\
&\le  2^{ (l+1)/l} \max(A_{l,1}(\rho),A_{l,2}(\rho)) N(f)
\end{align*}
We use the two assumptions and the two elementary inequalities $x+y \ge 2^{-1/l}(x^l + y^l)^{1/l}$ and $x+y \le 2 (x^l + y^l)^{1/l}$ for $x,y \ge 0$ and $l > 0 $.
\end{proof}
The important observation on Lemmas \ref{l-H1} and \ref{l-H2} is that the constants for properties \ref{h-lr} and \ref{h-inf} of $N$ does not depend on $c_1$ or $c_2$. We now prove the following pseudo-Poincar\'{e} inequality.
\begin{prop}[Pseudo-Poincar\'{e} inequality]\label{p-pp}
Let $(M,d,\mu)$ be a uniformly discrete, metric measure space satisfying \eqref{e-count},\eqref{e-vd} and \eqref{e-hom} and let $\mathcal{E}$ be a Dirichlet form whose jump kernel $J$ satisfies \ref{ljp}.
There exist constant $C_P>0$ such that
\begin{equation}
 \label{e-pp} \norm{f-f_r}_2^2 \le C_P r^\beta \mathcal{E}(f)
\end{equation}
 for all $f \in \ell^2(M,\mu)$ and for all $r >0$,
where $f_r(x) := \frac{1}{V(x,r)} \sum_{y \in B(x,r)} f(y) \mu_y$ is the $\mu$-average of $f$ in $B(x,r)$.
\end{prop}
\begin{proof}
We have
\begin{align*}
| f(x) - f_r(x)|^2 & =   \left| \frac{1}{\mu(B(x,r))} \int_{B(x,r)} ( f(x) - f(y)) d\mu(y) \right|^2 \\
& \le  \frac{1}{\mu(B(x,r))} \int_{B(x,r)} |f(x) - f(y)|^2 d\mu_c(y) \\
& \le  r^\beta \int_{B(x,r) \setminus \{x\}} \frac{|f(x) - f(y)|^2}{d(x,y)^\beta} \frac{d\mu_c(y)}{\mu(B(x,d(x,y)))} \\
\end{align*}
The second line above follows from Jensen's inequality.
Hence for $0 < r < \infty$, we have
\begin{equation} \label{e-pp1}
\| f - f_r \|_2  \le  r^{\beta/2} W_\beta (f)
\end{equation}
for all $f \in \ell^2(M,\mu)$, where $W_\beta$ denotes the Besov semi-norm
 \[
  W_\beta(f) = \left( \sum_{x,y \in M: x\neq y} \frac{\abs{f(x)- f(y)}^2}{d(x,y)^\beta V(x,d(x,y))} \mu_x \mu_y \right)^{1/2}.
 \]
Combining $d(x,y) \notin (0,a)$, \eqref{e-hom} and \ref{ljp}, there exists $C_2 >0$ such that
\begin{equation}\label{e-pp2}
 W_\beta(f)^2 \le (1+a^{-1})^\beta \left( \sum_{x,y \in M} \frac{\abs{f(x)- f(y)}^2}{(1+d(x,y))^\beta V(x,d(x,y))} \mu_x \mu_y \right) \le C_2\mathcal{E}(f).
\end{equation}
 The pseudo-Poincar\'{e} inequality \eqref{e-pp} follows from \eqref{e-pp1} and \eqref{e-pp2}.
\end{proof}
We are now ready to prove the Nash inequality \ref{N}.
\begin{prop}[Nash inequality]\label{p-nash}
Let $(M,d,\mu)$ be a uniformly discrete, metric measure space satisfying \eqref{e-count},\eqref{e-vd} and \eqref{e-hom} and let $\mathcal{E}$ be a Dirichlet form whose jump kernel $J$ satisfies \ref{ljp}.
Then $\mathcal{E}$ satisfies the Nash inequality \ref{N}.
\end{prop}

\begin{proof}
 Since $\E(\abs{f}) \le \E (f)$, it suffices to show \ref{N} for all $f \in \F$.
 We fix $\alpha> \max(\beta, \log_2 C_D)$ where $C_D$ is from \eqref{e-vd}. By  \eqref{e-count}, \eqref{e-hom} and \eqref{e-vc} there exists $C_1>0$ such that
 \begin{equation} \label{e-n1}
  \norm{f_r}_\infty \le  C_1 \frac{R^\alpha}{V_h(R)} r^{-\alpha} \norm{f}_1
 \end{equation}
for all $f \in \F$ and for all $0<r \le R$. Set  $\tau = 1 + \frac{\beta}{2 \alpha}$ and let $\lambda >0$. We now consider two cases $\lambda$ small and $\lambda$ large.

If $\lambda \le 3 C_1 \norm{f}_1 / V_h(R)$, by Markov inequality
$ \lambda^2 \mu(f \ge \lambda) \le \norm{f}_2^2$,  we have
\begin{equation}\label{e-n2}
 \lambda^{2\tau} \mu(f \ge \lambda) \le \norm{f}_2^2 \left( \frac{3 C_1 \norm{f}_2}{V_h(R)}\right)^{\beta/\alpha}.
\end{equation}
for all $f \in \F$ and for all $\lambda \le 3 C_1 \norm{f}_1 / V_h(R)$.

Now suppose $\lambda > 3 C_1 \norm{f}_1/V_h(R)$. Choose $0 < r < R$ such that
\begin{equation} \label{e-n3}
 \left( \frac{r}{R}\right)^\alpha = \frac{3 C_1 \norm{f}_1}{\lambda V_h(R)}.
\end{equation}
By \eqref{e-n1}, we have $\norm{f_r}_\infty \le \lambda/3$. Therefore by union bound and Proposition \ref{p-pp}, we have
 \begin{align*}
\mu(f \ge \lambda) &\le  \mu\left( \abs{f - f_r} \ge \lambda/2 \right) + \mu \left( \abs{f_r} \ge \lambda/2 \right) \\
& =   \mu\left( \abs{f - f_r} \ge \lambda/2\right) \\
& \le  (2/\lambda)^2 \| f - f_r \|_2^2 \\
& \le  C_P (2 /\lambda)^2 r^\beta \E(f)
\end{align*}
 Substituting $\lambda$ from \eqref{e-n3} yields,
 \begin{equation}\label{e-n4}
  \lambda^{2\tau} \mu(f \ge \lambda) \le 4 C_P \left( \frac{3 C_1}{V_h(R)}\right)^{\beta/\alpha} R^\beta \E(f) \norm{f}_1^{\beta/\alpha}
 \end{equation}
for all $f \in \F$ and for all $\lambda > 3 C_1 \norm{f}_1/V_h(R)$.
Combining \eqref{e-n2} and \eqref{e-n4}, we obtain the following weak Sobolev-type inequality: there exist constants $C_2,C_3>0$ such that
\begin{equation} \label{e-ws}
\sup_{\lambda >0 } \lambda^{2\tau} \mu(f \ge \lambda) \le C_2 \left( \frac{R^\alpha}{V_h(R)}\right)^{\beta/\alpha} \left( \E(f) + C_3 R^{-\beta} \norm{f}_2^2 \right) \norm{f}_1^{\beta/\alpha}
\end{equation}
for all $f \in \F$ and for all $R>0$. Set
\[
\frac{1}{q}= \frac{1}{2} - \frac{\beta}{2 \alpha}.
\]
Since $\beta < \alpha$,  we have $q>0$.

Define the semi-norm on $\F$ by $N_R(f)= \left( \frac{R^\alpha}{V_h(R)}\right)^{\beta/2\alpha}\left(\E(f) + C_3 R^{-\beta} \norm{f}_2^2\right)^{1/2}$.  Note that
\begin{equation} \label{e-n5}
\frac{1}{\sqrt{2}}   \le  \frac{N_R(f)} {\left(\frac{R^\alpha}{V_h(R)}\right)^{\beta/2\alpha}\left( \sqrt{\E(f)} + \sqrt{C_3} R^{-\beta/2} \norm{f}_2 \right)} \le 1.
\end{equation}
Therefore by Lemmas 2.1 and 7.1 of \cite{BCLS} and Lemmas \ref{l-H1} and \ref{l-H2}, we have that there exists $\rho>0$ and constants $A_\infty, A_q >0$ such that
\begin{align*}
 N_R ( (f-t)^+ \wedge s ) &\le A_\infty N_R (f) \\
\left( \sum_{k \in \Z} N_R(f_{\rho,k})^q \right)^{1/q} &\le A_q N_R(f)
\end{align*}
for all $f \in \F$, for all $R >0$ and for all $s,t \ge 0$.  Hence by \eqref{e-ws}, \eqref{e-n5}, Theorem 3.1 and \cite[Proposition 3.5]{BCLS}, there exists constant $C_4>0$ such that
\begin{equation}
 \norm{ f} _r \le  \left( C_4 \left(\frac{R^{\alpha}}{V_h(R)} \right)^{\frac{\beta}{\alpha}} \left( \E(f) + C_3 R^{-\beta} \norm{f}_2^2\right)\right)^{\vartheta/2} \norm{f}_s^ {1 - \vartheta}
\end{equation}
for all $f \in \F$, for all $R>0$, for all $r,s \in (0,\infty)$ and for all $\vartheta \in (0,1)$ such that
\[
 \frac{1}{r}= \frac{\vartheta}{q} + \frac{1-\vartheta}{s}.
\]
In particular, the choice $r=2, \vartheta= \alpha/(\alpha+\beta), s=1$ yields the desired Nash inequality \ref{N}.
 \end{proof}

 We conclude this section with a diagonal estimate on $h_n$. We need the following standard lemma.
 \begin{lem}[Folklore] \label{l-nondec}
  For any $ x\in M$, the map $n \mapsto h_{2n}(x,x)$ is a non-increasing function of $n \in \N$. Further
  \begin{equation} \label{e-cs}
   h_{2n}(x,y) \le \left( h_{2n}(x,x) h_{2n}(y,y) \right)^{1/2}
  \end{equation}
 \end{lem}
\begin{proof}
 Let $\ind_x$ denote the indicator function at $x$ and $\langle \cdot , \cdot \rangle$ denote the inner product in $\ell^2(M,\mu)$. Since $P$ is self-adjoint, we have
 \[
  h_{2n}(x,x) \mu_x^2 = \langle P^{2n} \ind_x, \ind_x \rangle = \langle P^{n} \ind_x,  P^{n} \ind_x \rangle = \norm{ P^{n} \ind_x}_2^2
 \]
for all $n \in \N$ and for all $x \in M$. Since $P$ is a contraction, we have \[ \norm{ P^{n+1} \ind_x}_2 \le \norm{ P^{n} \ind_x}_2 \] for all $n \in N$ and for all $x \in M$.
Combining the two facts concludes the proof of the first assertion.

For the second statement, we use $P$ is self-adjoint along with Cauchy-Schwarz to obtain
\begin{align*}
 h_{2n}(x,y) \mu_x \mu_y &= \langle P^{2n} \ind_x, \ind_y \rangle = \langle P^n \ind_x ,P^n \ind_y\rangle \\
 &\le \norm{P^n \ind_x }_2 \norm{P^n \ind_y}_2 = \left( h_{2n}(x,x) h_{2n}(y,y) \right)^{1/2} \mu_x \mu_y
\end{align*}
for all $x,y \in M$.
\end{proof}

 \begin{theorem}\label{t-dub}
Let $(M,d,\mu)$ be a uniformly discrete, metric measure space satisfying \eqref{e-count},\eqref{e-vd} and \eqref{e-hom} and let $\mathcal{E}$ be a Dirichlet form whose jump kernel $J$ satisfies \ref{ujp} and \ref{ljp}.
Then there exists a constant $C>0$ such that
\begin{equation}
 \label{e-d-dub} h_n(x,y) \le \frac{C}{V_h(n^{1/\beta})}
\end{equation}
for all $n \in \N^*$ and for all $x,y \in M$.
\end{theorem}
 \begin{proof}
 We first consider the case $x=y$.
  By \ref{ljp} and \eqref{e-count}, there exists $\kappa > 0$ such that $\inf_{x \in M} h_1(x,x) \ge \kappa$. Therefore by Chapman-Kolmogorov equation \eqref{e-ck} and \eqref{e-count}
  \begin{equation} \label{e-1s}
   h_{n+1}(x,y) \ge h_n(x,y) \kappa \mu_y \ge h_n(x,y) \kappa /C_\mu.
  \end{equation}
  for all $n \in \N$ and for all $x,y \in M$.
By Chebyschev inequality
\begin{equation} \label{e-cheb}
 \PP \left(  \abs{ N(3n/4) - (3n/4) } < n/4  \right) \ge 1- \frac{12}{n},
\end{equation}
where $N(\cdot)$ denotes the standard Poisson process.
First consider the case where $n$ is even and $n \ge 20$.  By \eqref{e-qdef}, we have
\begin{align}
\nonumber q_{3n/4}(x,x) &= \sum_{k=0}^\infty h_k(x,x) \PP( N(3n/4) = k) \\
\nonumber  & \ge  \sum_{k=n/2}^n  h_k(x,x) \PP( N(3n/4) = k) \\
\nonumber  & \ge  \frac{\kappa}{C_\mu} h_n(x,x) \PP \left(  \abs{ N(3n/4) - (3n/4) } \le n/4  \right) \\
  & \ge \frac{2\kappa}{5 C_\mu} h_n(x,x) \label{e-dis-cont}
\end{align}
The third line above follows from first assertion of Lemma \ref{l-nondec} and \eqref{e-1s} and the last line follows from \eqref{e-cheb} and $n \ge 20$.
By Propositions \ref{p-nash} and \ref{p-ultra} along with \eqref{e-vc}, \eqref{e-dis-cont} there exists $C_1>0$ such that
\begin{equation}\label{e-nbigeven}
 h_n(x,x) \le \frac{C_1}{V_h(n^{1/\beta})}
\end{equation}
for all even $n \ge 20$.

The case $n$ is odd and $n \ge 19$ follows from \eqref{e-1s} and \eqref{e-nbigeven}. The case $n < 19$ follows from the observation that
\[
 \sup_{x} h_n(x,x) \le \sup_{x,y} h_1(x,y)
\]
for all $n \in \N$,  along with \ref{ujp}, \eqref{e-count} and \eqref{e-vc}. Combining all the cases, there exists $C_2>0$ such that
\begin{equation} \label{e-diag-ub1}
  h_n(x,x) \le \frac{C_2}{V_h(n^{1/\beta})}
\end{equation}
for all $x \in M$ and for all $n \in \N^*$.

By \eqref{e-cs} and \eqref{e-diag-ub1}, we have
\begin{equation*} \label{e-diag-ub2}
  h_n(x,y) \le \frac{C_2}{V_h(n^{1/\beta})}
\end{equation*}
for all $x,y \in M$ and for all $n \in \N^*$ and $n$ even.  If $n$ is odd, the desired estimate follows from \eqref{e-1s}.
 \end{proof}
\section{Upper bound on continuous time kernel}\label{s-ub}
In this section, we prove off-diagonal upper bound on $q_t$ using the method of \cite{BGK}.
As a consequence of this upper bound on $q_t$, we obtain estimates on hitting times and exit times for $X_n$.

The idea behind the approach of \cite{BGK} is to use Meyer's construction \cite{Mey} to split the jump kernel into small and large jumps and apply Davies' method for the smaller jumps (see \cite[Section 3]{BGK}).
We need the following estimates to show the upper bound on $q_t$.
\begin{lemma}\label{l-te} Let $(M,d,\mu)$ be a uniformly discrete, metric measure space satisfying \eqref{e-count},\eqref{e-vd} and \eqref{e-hom}.
 There exist constants $C_1,C_2 >0$ such that
 \begin{align}
  \label{e-te1} \sum_{y \in B(x,r)^c} \frac{\mu_y}{V_h(d(x,y)) d(x,y)^\beta} & \le  C_1 r^{-\beta} \\
  \label{e-te2} \sum_{y \in B(x,r)} d(x,y)^{2-\beta}\frac{\mu_y}{V_h(d(x,y))} &\le C_2 r^{2-\beta}
 \end{align}
for all $x \in M$, $r >0$ and $\beta \in (0,2)$.
\end{lemma}
\begin{proof}
For \eqref{e-te1}, observe that
 \begin{align*}
 \sum_{y \in B(x,r)^c} \frac{\mu_y}{V_h(d(x,y)) d(x,y)^\beta}  &\le \sum_{n=1}^{\infty} \sum_{y \in B(x,2^n r) \setminus B(x,2^{n-1}r)} \frac{\mu_y}{V_h(d(x,y)) d(x,y)^\beta} \\
&\le  \sum_{n=1}^\infty C_h \frac{V_h(2^n r)} { V_h(2^{n-1}r) \left( 2^{(n-1)}r \right)^{ \beta}}\\
&\le C_1 r^{-\beta}.
\end{align*}
We used \eqref{e-hom} in the second line above and \eqref{e-vd} in the last line. For  \eqref{e-te2}, note that
\begin{align*}
\lefteqn{\sum_{y \in B(x,r)} d(x,y)^{2-\beta} \frac{\mu_y}{V_h(d(x,y))}  }\\
&= \sum_{y \in B(x,r) \setminus B(x,a)} d(x,y)^{2-\beta} \frac{\mu_y}{V_h(d(x,y))} \\
&\le \sum_{n=1}^{\lceil \log_2 \left(r/a\right) \rceil} \sum_{y \in B(x,2^n a) \backslash B(x,2^{n-1}a)} d(x,y)^{2-\beta}\frac{\mu_y}{V_h(d(x,y))} \\
& \le   \sum_{n=1}^{\lceil \log_2\left( r/a \right) \rceil} C_h a^{2-\beta} \frac{2^{n(2-\beta)} V_h(2^n a )} { V_h(2^{n-1}a)}\\ \\
&\le C_2 r^{2-\beta}.
\end{align*}
In the second line above, we used that $(M,d,\mu)$ is uniformly discrete, in the fourth line we used  \eqref{e-hom} and in the last line we used \eqref{e-vd}.
\end{proof}
We now obtain the following off-diagonal estimate using Meyer's splitting of jump kernel and Davies' method as outlined in \cite{BGK}.
The main difference from \cite{BGK} is that the Nash inequality \ref{N} and volume growth we use are more general.
\begin{theorem} \label{t-ubq}
 Let $(M,d,\mu)$ be a uniformly discrete, metric measure space satisfying \eqref{e-count},\eqref{e-vd} and \eqref{e-hom} and
 let $\mathcal{E}$ be a Dirichlet form whose jump kernel $J$ with respect to $\mu$ satisfies \ref{ujp} and \ref{ljp} for some $\beta \in (0,2)$.
Then there exists $C>0$, such that the transition density $q_t$ satisfies
\begin{equation} \label{e-qub}
 q_t(x,y) \le C \left( \frac{1}{V_h(t^{1/\beta})} \wedge \frac{t}{V_h(d(x,y)) (1+d(x,y))^\beta} \right)
\end{equation}
for all $t >0$ and for all $x,y \in M$.
\end{theorem}
\begin{proof}

By Lemma \ref{l-te} and \ref{ujp}, there exists $C_2,C_3 >0$ such that
\begin{eqnarray}
\label{e-qub2} \sum_{y \in B(x,r)^c} J(x,y) \mu_y &\le& C_2 r^{-\beta} \\
 \label{e-qub3} \sum_{y \in B(x,r)} d(x,y)^2 J(x,y)\mu_y &\le& C_3 r^{2-\beta}
\end{eqnarray}
for all $r >0$ and for all $x \in M$.

Let $J_K$ denote the jump density $J_K(x,y):= J(x,y) \ind_{d(x,y) \le K}$ and let $q^K_t(x,y)$ denote the corresponding transition density with respect to $\mu$.
Set $\E_K$, the corresponding Dirichlet form
\[
 \E_K(f,f) := \frac{1}{2} \sum_{x,y \in M} \abs{f(x) -f(y)}^2 J_K(x,y) \mu_x \mu_y.
\]
Note that
\begin{align}
\nonumber \E(f,f)-\E_K(f,f) &= \frac{1}{2} \sum_{x,y \in M: d(x,y) > K} \abs{f(x) -f(y)}^2 J(x,y) \mu_x \mu_y\\
\nonumber & \le  \sum_{x,y \in M: d(x,y) > K} \left( f(x)^2 +f(y)^2 \right) J(x,y) \mu_x \mu_y\\
\label{e-qub4} & \le  2 C_2 \norm{f}_2^2 K^{-\beta}.
\end{align}
In the last step above, we used symmetry of $J$ and \eqref{e-qub2}.
By Nash inequality (Proposition \ref{p-nash}) and \eqref{e-qub4}, there exists $\alpha,C_4,C_5>0$ such that
\[
 \norm{f}_2 \le C_4 \left(\left(\frac{K^{\alpha}}{V_h(K)}\right)^{\frac{\beta}{\alpha}}\mathcal{E}_K(f,f) + C_5 K^{-\beta} \|f\|_2^2  \right)^{\frac{\alpha}{2(\alpha+\beta)}} \| f \|_1^ {\frac{\beta}{\alpha+\beta}}
\]
for all $K>0$ and for all $f \in \ell^1(M,\mu)$. By Davies' method (\cite[Theorem 3.25]{CKS}) as described in \cite[Theorem 1.4]{BGK}, there exists a constant $C_6 >0$ such that
\begin{equation}\label{e-qub5}
 q_t^K(x,y) \le C_6 \frac{K^\alpha}{V_h(K)} t^{-\alpha/\beta} \exp \left( C_5 t K^{-\beta} - E_K(2t,x,y) \right)
\end{equation}
for all $x,y \in M$, for all $t >0$ and for all $K >0$, where $E_K$ is given by
\begin{align*}
 \Gamma_K(\psi)(x)&= \sum_{y \in M} \left(e^{(\psi(x) -\psi(y))}-1\right)^2 J_K(x,y) \mu_y,\\
\Lambda_K(\psi)^2&=  \norm{\Gamma_K(\psi)}_\infty \vee \norm{\Gamma_K(-\psi)}_\infty, \\
E_K(t,x,y) &= \sup \{ \abs{\psi(x)-\psi(y)}-t \Lambda_K(\psi)^2 : \psi \in C_c(M) \}.
\end{align*}
By Proposition \ref{p-ultra} and Cauchy-Schwarz inequality \[q_t(x,y) \le \left( q_t(x,x) q_t(y,y)\right)^{1/2},\] it suffices to show that there exists $C_1>0$ such that
\begin{equation}
\label{e-qub1} q_t(x,y) \le C_1 \frac{t}{V_h(d(x,y)) (1+ d(x,y))^{\beta}}
\end{equation}
for all $x,y \in M$ such that $d(x,y)^\beta \ge \theta^\beta t$ where $\theta= 3(\alpha+\beta)/ \beta$.

Let $x,y \in M$ be such that $K^\beta \ge t$ where $K=d(x,y)/\theta$.
Define $\psi(z) = \lambda( d(x,y) - d(x,z))_+$.
Using $\abs{e^t -1}^2 \le t^2 e^{2\abs{t}}$, $\abs{\psi(z)-\psi(w)} \le \lambda d(w,x)$  and \eqref{e-qub3}, we get
\begin{align*}
\Gamma_K(e^\psi)(z) &= \sum_{w \in B(z,K)} \left( e^{\psi(z)-\psi(w)} -1\right)^2J_K(z,w) \mu_w  \\
&\le e^{2 \lambda K} \lambda^2 \sum_{w \in B(z,K)} d(z,w)^2 J_K(z,w) \mu_w\\
& \le  C_3 (\lambda K )^2 e^{2 \lambda K} K^{-\beta} \\
&\le  C_3  e^{3 \lambda K} K^{-\beta}
\end{align*}
for all $z \in M$ and for all $\lambda,K>0$.
It follows that
\[
 -E_K(2t,x,y) \le -\lambda d(x,y) + 2 C_3 t e^{3 \lambda K} K^{-\beta}.
\]
We fix
\[
 \lambda= \frac{1}{3K} \log \left(\frac{K^\beta}{t} \right),
\]
so that
\begin{align*}
-E_K(2t,x,y) &\le \frac{-d(x,y)}{3K} \log \left(\frac{ K^\beta}{t} \right) + 2 C_3 t K^{-\beta} \frac{K^\beta}{t} \\
&\le 2C_3 - \left( \frac{\alpha +\beta}{\beta} \right) \log \left( \frac{K^\beta}{t} \right).
\end{align*}
By \eqref{e-qub5} and $K^\beta \ge t$,
\begin{align}
 \nonumber q_t^K(x,y) &\le C_6 \frac{K^\alpha}{V_h(K)} t^{-\alpha/\beta} \exp\left( C_5 +2 C_3 \right) \left( \frac{t}{K^\beta}\right)^{(\alpha+\beta)/\beta} \\
\label{e-qub6} &\le C_7 \frac{t}{V_h(K) K^\beta} \le C_8 \frac{t}{V_h(d(x,y)) (1+d(x,y))^{1+\beta}}
\end{align}
In the last step, we used that $d(x,y) \ge a$, $K=d(x,y)/\theta$ and \eqref{e-vc}. By \ref{ujp}, $d(x,y) >a$ and \eqref{e-vc}, there exists $C_9>0$ such that
\[
J(x,y)-J_K(x,y) \le C_9  \frac{1}{ V_h(d(x,y)) (1+d(x,y))^\beta}
\]
for all $x,y \in M$.
Therefore by \cite[Lemma 3.1(c)]{BGK}, there exists $C_1>0$ such that
\[
 q_t(x,y) \le C_{1} \frac{t}{V_h(d(x,y)) (1+d(x,y))^\beta}
\]
for all $x,y \in M$, for all $t>0$  such that $d(x,y)^\beta \ge \theta^\beta t$ which proves \eqref{e-qub1} and hence \eqref{e-qub}.
\end{proof}
\subsection{Exit time and Hitting time estimates}
 In this subsection, we  apply Theorem \ref{t-ubq} to estimate hitting time and exit time of balls for the discrete time Markov chain $X_n$ and the corresponding continuous time chain $Y_t$.
 \begin{notn}
  We denote  exit time and hitting time of the ball $B(x,r)$ by
  \begin{align*}
   \tau_X(x,r) &= \inf \{ k : X_k \notin B(x,r) \} \\
    \tau_Y(x,r) &= \inf \{ t : Y_t \notin B(x,r) \} \\
    T_X(x,r) &= \inf \{ k : X_k \in B(x,r) \} \\
    T_Y(x,r) &= \inf \{ t : Y_t \in B(x,r) \}
  \end{align*}
for all $x \in M$ and for all $r>0$.
 \end{notn}
We start with exit and hitting time estimates for continuous time Markov chain $Y_t$.
 \begin{prop}
  \label{p-exY} Under the assumptions of Theorem \ref{t-ubq}, there exists $C_1>0$ such that
  \[
   \PP^x \left( \tau_Y(x,r) \le t \right) \le C_1 \frac{t}{r^\beta}
  \]
for all $x \in M$ and for all $t,r >0$.
 \end{prop}
\begin{proof}
 By Theorem \ref{t-ubq}, there exists $C_2>0$ such that
 \begin{align}
\PP^x \left( d(Y_t,x) \ge r \right) &= \sum_{y \in B(x,r)^c} q_t(x,y) \mu_y \nonumber \\
&\le Ct \sum_{y \in B(x,r)^c} \frac{\mu_y}{V_h(d(x,y))d(x,y)^{\beta} } \nonumber \\
 \label{c1} & \le  C_2 \frac{t}{r^\beta}
\end{align}
for all $x \in M$ and for all $r,t >0$.
The last line follows \eqref{e-te1} of Lemma \ref{l-te}.

Set $\tau=\tau_Y(x,r)$. There exists $C_1>0$ such that
\begin{align*}
\mathbb{P}^x(\tau \le t) &\le \mathbb{P}^x\left( \tau \le t , d(Y_{2t},x) \le r/2 \right) + \mathbb{P}^x \left( d(Y_{2t},x) > r/2\right) \\
&\le \mathbb{P}^x\left( \tau \le t, d(Y_{2t},Y_\tau) \ge r/2 \right) + 8 C_2 t/r^\beta \\
&= \mathbb{E}^x \left( \ind_{\tau \le t} \mathbb{P}^{Y_\tau} \left( d (Y_{2t-\tau},Y_0) \ge r/2 \right) \right) + 8C_2 t /r^\beta \\
& \le \sup_{y \in B(x,r)^c} \sup_{s \le t} \mathbb{P}^y\left( d (Y_{2t -s },y) \ge r/2 \right) + 8 C_2 t / r^\beta \\
& \le  C_1 t/r^\beta
\end{align*}
for all $x \in M$ and for all $r,t> 0$.
The second and fifth lines follow from (\ref{c1}) and the third line above follows from strong Markov property.
\end{proof}
Similarly, we have the following estimate for the hitting time $T_Y$.
\begin{lem} \label{l-hitY}
Under the assumptions of Theorem \ref{t-ubq}, there exists $C_1>0$ such that
 \[
  \PP^x \left( T_Y(y, t^{1/\beta}) \le t \right) \le C_1 \left( \frac{t V_h(t^{1/\beta})}{V_h(d(x,y)) d(x,y)^\beta} \right)
 \]
for all $x,y \in M$ and for all $t>0$.
\end{lem}
\begin{proof}
 By Proposition \ref{p-exY}, there exists $C_2>0$ such that
 \begin{equation}\label{e-hY1}
  P^z\left( \tau_Y(z, C_2 t^{1/\beta}) > t\right) \ge \frac{1}{2}
 \end{equation}
for all $z \in M$ and for all $t>0$. By \eqref{e-vc}, it suffices to consider the case $ d(x,y) > 2(1+C_2)t^{1/\beta} $.

Set $S=T_Y(y, t^{1/\beta})$.
By \eqref{e-hY1} and strong Markov property,
\[
 \PP^x\left(S \le t, \sup_{S \le k \le t+ S} d(Y_k,Y_S) \le C_2 t^{1/\beta} \right) \ge \frac{1}{2} \PP^x(S \le t)
\]
for all $x,y \in M$ and $t>0$.
Therefore
\begin{align*}
\mathbb{P}^x(S \le t) &\le 2 \mathbb{P}^x \left(S\le t , \sup_{S \le k \le S+t }  \left| Y_k - Y_S\right| \le C_2 t^{1 / \beta}  \right) \\
& \le  2\mathbb{P}^x \left( Y_{t} \in B(y,(1+C_2) t^{1 / \beta}) \right) \\
&= 2\sum_{z \in B(y,(1+C_2) t^{1 / \beta}) } q_{t} (x, z) \mu_z \\
& \le  2C t \sum_{ z\in B(y,(1+C_2) t^{1 / \beta}) } \frac{1}{ V_h(d(x,z)) d(x,z)^{\beta}} \mu_z \\
& \le  2C t \sum_{z \in B(y,(1+C_2) t^{1 / \beta}) } \frac{1}{ V_h(d(x,y)/2) (d(x,y) /2)^{\beta}}\mu_z \\
& \le  2^{1+\beta} C t   \frac{V_h( (1+C_2) t^{1/ \beta})}{V_h(d(x,y)/2) d(x,y)^{\beta}} \\
& \le  C_1  \frac{t  V_h(  t^{1/ \beta})}{ d(x,y)^{\beta} V_h(d(x,y))} \\
\end{align*}
for all $x ,y \in M$ and for all $t>0$.
The fourth line above follows from Theorem \ref{t-ubq}. The fifth line follows from $d(x,z) \ge d(x,y)/2$ which is a consequence of  $ (1+C_2) t^{1 / \beta} \le \frac{d(x,y)}{2}$ and triangle inequality.
 The last line follows from \eqref{e-vc}.
\end{proof}
Now we prove similar estimates for $X_n$. The strategy is to compare the behavior of $X_n$ with $Y_t$ using the equation $Y_t=X_{N(t)}$, where $N(t)$ is a standard Poisson process independent of $(X_k)_{k \in \N}$.
Define $T_k$ as the arrival times of Poisson process defined by  $N(t)=k$ for all $t \in [T_k,T_{k+1})$ and for all $k \in \N$. Then $T_k$ is an exponential random variable with mean $k$ and independent of $(X_n)_{n \in \N}$.
\begin{prop}\label{p-hitx}
Under the assumptions of Theorem \ref{t-ubq}, there exists $C_1>0$  such that
\[
 \PP^x \left( T_X(y,n^{1/\beta}) \le n \right) \le C_1 \frac{ n V_h (n^{1/\beta})}{ V_h(d(x,y))(1+d(x,y))^\beta}
\]
for all $ n \in \N$ and for all $x,y \in M$ with $x \neq y$.
\end{prop}
\begin{proof}
 It suffices to consider the case $n \ge 1$. By Markov inequality $\PP(T_n > 2n) \le 1/2$. Therefore by independence of  $(X_n)_{n \in \N}$ and the arrival time $T_n$, we have
 \begin{align*}
  \frac{1}{2} \PP^x ( T_X(y,n^{1/\beta}) \le n ) &\le \PP^x \left( T_X(y,n^{1/\beta}) \le n , T_n \le 2n\right) \\
& \le  \PP^x \left( T_Y(y,n^{1/\beta}) \le 2n\right)  \\
& \le  C_2 \frac{ 2n V_h( (2n)^{1/\beta})}{d(x,y)^\beta V_h(d(x,y))}
 \end{align*}
 for all $x,y \in M$ with $x \neq y$ and for all $ n \in \N^*$.
The last line above follows from Lemma \ref{l-hitY}. The conclusion then follows from \eqref{e-vc}.
\end{proof}
We conclude the section, with an exit time estimate for $X_n$.
\begin{prop} \label{p-exitX}
 Under the assumptions of Theorem \ref{t-ubq}, there exists $\gamma>0$ such that
 \begin{equation}\label{e-exitX}
   \PP^x\left( \max_{0\le k \le \lfloor \gamma r^\beta \rfloor} d(X_k,x) > r/2 \right) \le 1/4
 \end{equation}
for all $x \in M$ and for all $r >0$.
\end{prop}
\begin{proof}
Choose $\gamma_1>0$ such that $2^{\beta+1} C_1 \gamma_1 = 1/8$, where $C_1$ is the constant from Proposition \ref{p-exY}. By Proposition \ref{p-exY},
\[
 \PP^x\left( \sup_{s \le 2 \gamma_1 r^\beta} d(Y_s,x) > r/2\right) = \PP^x \left( \tau_Y(x,r/2) \le 2\gamma_1 r^\beta \right) \le 2^{\beta+1} C_1 \gamma_1 \le 1/8.
\]
for all $x \in M$ and for all $r>0$.
 Therefore
\begin{eqnarray}
\nonumber \lefteqn{\mathbb{P}^x \left( \max_{s \le \lfloor \gamma_1 r^\beta \rfloor} d(X_k,x) > r/2\right) } \\
\nonumber &=&  \mathbb{P}^x \left( \max_{s \le \lfloor \gamma_1 r^\beta \rfloor} d(X_k,x) > r/2,T_{\lfloor \gamma_1 r^\beta \rfloor}  \le 2 \lfloor \gamma_1 r^\beta \rfloor   \right) \\
\nonumber& & + \mathbb{P}^x \left( \max_{s \le \lfloor \gamma_1 r^\beta \rfloor} d(X_k,x) > r/2,T_{\lfloor \gamma_1 r^\beta \rfloor}  > 2 \lfloor \gamma_1 r^\beta \rfloor   \right) \\
\nonumber &\le& \mathbb{P}^x \left( \sup_{s \le 2 \gamma_1 r^\beta} d(Y_s,x) > r/2 \right) \\
\nonumber &&+ \mathbb{P}^x \left( T_{\lfloor \gamma_1 r^\beta \rfloor}  - \lfloor \gamma_1 r^\beta \rfloor > \lfloor \gamma_1 r^\beta \rfloor \right) \\
\label{e-exitX1} & \le & \frac{ 1}{8} + \frac{1}{\gamma_1 r^\beta}
\end{eqnarray}
for all $x \in M$ and for all $r >0$.
In the last line above, we used Markov's inequality $P(\left| X\right| >a ) \le \frac{\mathbb{E}X^2}{a^2}$ for $X=T_{\lfloor \gamma_1 r^\beta \rfloor}  - \lfloor \gamma_1 r^\beta \rfloor$.
Fix $r_0$ so that $\gamma_1 r_0^\beta = 8$ and choose $\gamma \in (0, \gamma_1/8)$, so that $\gamma r_0^\beta <1$.

If $r < r_0$ , then
$\gamma r^\beta <1$ and $\mathbb{P}^x \left( \max_{s \le \lfloor \gamma r^\beta \rfloor} d(X_k,x) > r/2\right) =0 \le 1/4$.

If $r \ge r_0$, then by \eqref{e-exitX1}
\[
\mathbb{P}^x \left( \max_{s \le \lfloor \gamma r^\beta \rfloor} d(X_k,x) > r/2\right) \le  \mathbb{P}^x \left( \max_{s \le \lfloor \gamma_1 r^\beta \rfloor} d(X_k,x) > r/2\right) \le \frac{1}{4}.
\]
Combining the cases $r<r_0$ and $r \ge r_0$ gives the desired result.
\end{proof}
\section{Parabolic Harnack inequality}\label{s-phi}
In this section, we follow an iteration argument due to Bass and Levin \cite{BL} to prove a parabolic Harnack inequality.

Let $\mathcal{T} = \{0,1,2,\ldots \} \times M$ denote the discrete space-time. We will study the
$\mathcal{T}$-valued Markov chain $(V_k,X_k)$, where the time component $V_k=V_0+k$ is deterministic and the space component  $X_k$ is same as the discrete time Markov chain with transition density $J$ with respect to $\mu$. We write
$\mathbb{P}^{(j,x)}$ for the law of $(V_k,X_k)$ started at $(j,x)$.
Let $\mathcal{F}_j = \sigma\left( (V_k,X_k): k \le j \right)$ denote the natural filtration associated with $(V_k,X_k)$. Given $D \subset \mathcal{T}$, we denote by $\tau_D$ the exit time
\[
 \tau_D := \min \{ k \ge 0: (V_k,X_k) \notin D \}.
\]

\begin{definition}
A bounded
function $u(k,x)$ on $\T$ is said to be \emph{parabolic} on $D \subset \T$
if $u(V_{k \wedge \tau_D},X_{k \wedge \tau_D})$ is a martingale. In other words, $u$ satisfies the discrete time backwards heat equation
\[
 u_{n}(x)= P u_{n+1} (x)
\]
for all $(n,x) \in D$, where $u_k(x)=u(k,x)$ for all $(k,x) \in \T$.
\end{definition}
It is immediate that if $D_1 \subset D_2$  and if $u$ is parabolic on $D_2$, then $u$ is parabolic on $D_1$. The main example of parabolic function that we have in mind is the heat kernel as demonstrated in the following lemma.

 \begin{lem} \label{l-parab}
For each $n_0$ and $x_0 \in M$, the function $q(k,x)= h_{n_0 -k}(x,x_0)=h_{n_0 -k}(x_0,x) $ is parabolic on $\{ 0,1,\ldots,n_0\} \times M$.
\end{lem}
\begin{proof}
\begin{eqnarray*}
\mathbb{E}\left[ q(V_{k+1},X_{k+1}) \vert \mathcal{F}_k\right] &=& \mathbb{E}\left[ h_{n_0 -V_{k+1}} (X_{k+1},x_0) \vert \mathcal{F}_k\right]  \\
&=& \mathbb{E}^{(V_k,X_k)} \left[ h_{n_0-V_1}( X_1,x_0 ) \right] \\
&=& \sum_z h_1(X_k,z)h_{n_0-V_k-1}(z,x_0) \\
&=& h_{n_0-V_k}(X_k,x_0) = q(V_k,X_k).
\end{eqnarray*}
The second equation follows from Markov property and last equation follows from  Chapman-Kolmogorov equation \eqref{e-ck}.
\end{proof} For $(k,x) \in \T$ and $A \subset \T$, define $N_A(k,x) := \PP^{(k,x)} \left( X_1 \in A(k+1) \right)$ if $(k,x) \notin A$ and $0$ otherwise.
\begin{lem}
 \label{l-mart} For the $\T$-valued Markov chain $(V_k,X_k)$, let $A \subset \T$ and
\[J_n= \ind_A(V_n,X_n) - \ind_A(V_0,X_0) -\sum_{k=0}^{n-1} N_A(V_k,X_k).
\]
Then $J_{n \wedge T_A}$ is a martingale.
\end{lem}
\begin{proof}
 We have
\begin{eqnarray*}
\lefteqn{\mathbb{E}\left[ J_{(k+1)\wedge T_A} - J_{k \wedge T_A} | \mathcal{F}_k \right] } \\
 &=&
\mathbb{E} [ \ind_A \left( V_{(k+1) \wedge T_A},X_{(k+1)\wedge T_A} \right) - \ind_A \left( V_{k \wedge T_A},X_{k \wedge T_A} \right) 
- N_A(V_{k \wedge T_A},X_{k\wedge T_A}) | \mathcal{F}_k ].
\end{eqnarray*}
On the event $\{T_A \le k \}$, this is 0. If $T_A > k$, this is equal to
\begin{eqnarray*}
\lefteqn{\mathbb{P}^{(V_k,X_k)} \left( (V_1,X_1) \in A \right) - N_A(V_k,X_k)} \\ &=& \mathbb{P}^{(V_k,X_k)} \left( X_1 \in A(V_k +1) \right) - N_A(V_k,X_k) = 0.
\end{eqnarray*}
\end{proof}
The next three technical lemmas are needed for the proof of parabolic Harnack inequality. They compare various hitting and exit times for the $\T$-valued Markov chain $(V_k,X_k)$.

We introduce a few notations.
Let $\gamma$ be a constant satisfying \eqref{e-exitX} from Proposition \ref{p-exitX}.
Define \[Q(k,x,r) := \{ k,k+1,\ldots,k+\lfloor \gamma r^\beta \rfloor \} \times B(x,r). \]
For the $\T$-valued Markov chain $(V_k,X_k)$ defined above, we denote the exit time of $Q(0,x,r)$ by \[\tau(x,r):= \min\{ k : (V_k,X_k) \notin Q(0,x,r). \]
Note that $\tau(x,r) \le \lfloor \gamma r^\beta \rfloor +1$ is a bounded stopping time.

For $A \subset \T$ and $k \in \N$, we set $A(k):= \{ y \in M : (k,y) \in A \}$.

Given a set $A \subset T$, we denote the hitting time by $T_A= \min \{ k : (V_k,X_k) \in A \}$ and the cardinality of $A$ by $\abs{A}$.
\begin{lem} \label{l-t1}
 Under the assumptions of Theorem \ref{t-ubq}, there exists $\theta_1>0$ such that
 \[
  \PP^{(0,x)} \left( T_A < \tau(x,r) \right) \ge \theta_1 \frac{\abs{A}}{V_h(r)r^\beta}
 \]
for all $x \in M$, for all $r>0$ and for all $A \subset Q(0,x,r)$ satisfying $A(0) = \emptyset$.
\end{lem}

\begin{proof}
 Since $A(0)=\emptyset$ and $A \subset Q(0,x,r)$, it suffices to consider the case $\gamma r^\beta \ge 1$. We abbreviate $\tau(x,r)$ by $\tau$.  Since $A \subset Q(0,x,r)$, $T_A \neq \tau$.

 By \eqref{e-hom}, \eqref{e-vd}, there exists $C_1>0$ such that
 \[
  \frac{\abs{A}}{V_h(r)r^\beta} \le \frac{\abs{Q(0,x,r)}}{V_h(r)r^\beta} \le C_1
 \]
for all $x \in M$ and for all $r>0$. Therefore if $\PP^{(0,x)}(T_A \le \tau ) \ge 1/4$, we are done.

We may assume, without loss of generality that $\PP^{(0,x)}(T_A \le \tau) <1/4$. Define the stopping time $S= T_A \wedge \tau$. By Lemma \ref{l-mart} and optional stopping theorem, we have
\begin{equation}\label{e-t1-1}
  \mathbb{P}^{(0,x)}(T_A < \tau_r)=\mathbb{E}^{(0,x)}\ind_A(S,X_S) \ge \mathbb{E}^{(0,x)} \sum_{k=0}^{S-1}N_A(k,X_k).
\end{equation}
By \eqref{e-count} and \ref{ljp} there exists $\kappa >0$ such that $p(x,x) > \kappa$ for all $x \in M$.

There exist $c_1,c_2>0$ such that,
\begin{eqnarray}
\nonumber N_A(k,w)&=&\mathbb{P}^{(k,w)}(X_1 \in A(k+1))\\
\nonumber & \ge& \sum_{y \in A(k+1),y \neq w} \frac{c_1}{V_h(d(w,y)) d(w,y)^{\beta}} + \ind_{A(k+1)}(w) \kappa\\
\label{e-t1-2} &\ge& \frac{c_2}{V_h(r) r^\beta} \left|A(k+1) \right|
\end{eqnarray}
for all $x \in M$, $r>0$ and for all $(k,w) \in Q(0,x,r) \setminus A$. In the second line above we used, \ref{ljp} and that $d$ is uniformly discrete. For the last line, we used $d(w,y) \le 2r$, \eqref{e-vd}, \eqref{e-count} and $\gamma r^\beta \ge 1$.

On the event that $S \ge \lfloor \gamma r^\beta \rfloor$,  by $A(0)=\emptyset$, $A \subset Q(0,x,r)$, \eqref{e-t1-1} and \eqref{e-t1-2}, there exists $c_3>0$ such that
\[
\sum_{k=0}^{S-1}N_A(k,X_k) \ge c_3 \frac{\abs{A}}{V_h(r)r^\beta}.
\]
Since $\tau \le \lfloor \gamma r^\beta \rfloor +1$ and $T_A \neq \tau$, we have
\begin{eqnarray*}
\mathbb{E}^{(0,x)} \ind_A(S,X_S)&\ge& c_3 \frac{\left|A\right|}{V_h(r) r^{\beta}} \mathbb{P}^x (S \ge \lfloor \gamma r^\beta \rfloor) \\
&\ge & c_3 \frac{\left|A\right|}{V_h(r) r^{\beta}} \left(1 - \mathbb{P}^x (T_A \le \tau ) - \mathbb{P}^x( \tau \le \lfloor \gamma r^\beta \rfloor )\right) \\
& \ge & c_3 \frac{\left|A\right|}{2 V_h(r) r^{\beta}}.
\end{eqnarray*}
The second line follows from the union bound by observing $ \{ S < \lfloor \gamma r^\beta \rfloor \}  \subseteq \{ T_A \le \tau \} \cup \{ \tau \le \lfloor \gamma r^\beta \rfloor) \}$.
The last inequality is due to our choice of $\gamma$ satisfying \eqref{e-exitX}
and the assumption that  $\mathbb{P}^x (T_A \le \tau ) < 1/4$.
\end{proof}
Define the set $U(k,x,r)=\{k\} \times B(x,r)$.
\begin{lem}\label{l-t2}
 Under the assumptions of Theorem \ref{t-ubq}, there exists $\theta_2 >0$ such that, for all $(k,x) \in Q(0,z,R/2)$, for all $r \le R/2$ and for all $k \ge \lfloor \gamma r^\beta \rfloor +1$, we have
 \[
  \PP^{(0,z)} \left(T_{U(k,x,r)} < \tau(z,R) \right) \ge \theta_2 \frac{V_h(r) r^{\beta}}{V_h(R) R^{\beta}}
 \]
for all $z \in M$ and for all $R >0$.
\end{lem}
\begin{proof}
 Let $Q'= \{ k,k-1,\ldots,k-\lfloor \gamma r^\beta \rfloor \} \times B(x,r/2)$. By triangle inequality $B(x,r/2) \subset B(z,R)$. Therefore $Q' \subset Q(0,z,R)$ and $Q'(0) =\emptyset$.
 By Lemma \ref{l-t1} and \eqref{e-vc}, there exists $c_1>0$ such that
\[
\PP^{(0,z)} (T_{Q'}< \tau(z,R)) \ge c_1 \frac{V_h(r) r^{\beta}}{ V_h(R) R^{\beta}}
\]
for all $z \in M$, for all $R>0$ and for all $r \in (0, R]$.

By the choice of $\gamma$ satisfying \eqref{e-exitX}, starting at a point in $Q'$ there is a probability of at least $3/4$ that the chain stays in $B(x,r)$ for at least time
$\lfloor \gamma r^\beta \rfloor$. By strong Markov property, there is a probability of at least $\frac{3}{4}c_1 \frac{V_h(r) r^{\beta}}{ V_h(R) R^{\beta}}$ that
the chain hits $Q'$ before exiting $Q(0,z,R)$  and stays within $B(x,r)$ for an additional time $\lfloor \gamma r^\beta \rfloor$, hence hits $U(k,x,r)$ before exiting $Q(0,z,R)$.
\end{proof}
\begin{lem} \label{l-t3}
Suppose $H(k,w)$ is nonnegative and $0$ if $w \in B(x,2r)$.  Under the assumptions of Theorem \ref{t-ubq}, there exists $\theta_3$ (not depending on $x,r,H$) such that
\begin{equation} \label{e-t3}
\mathbb{E}^{(0,x)} \left[ H(V_{\tau(x,r)},X_{\tau(x,r)})\right] \le \theta_3 \mathbb{E}^{(0,y)} \left[  H(V_{\tau(x,r)},X_{\tau(x,r)}) \right]
\end{equation}
for all $y \in B(x,r/2)$.
\end{lem}
\begin{proof}
 By the linearity of expectation and the inequality $1 \le \tau(x,r) \le \lfloor \gamma r^\beta \rfloor +1$,
 it suffices to verify \eqref{e-t3} for  indicator functions $H= \ind_{(k,w)}$ for all $x \in M$, for all $r >0$, for all $y \in B(x,r/2)$, for all $w \notin B(x,2r)$ and for all $1 \le k \le \lfloor \gamma r^\beta \rfloor +1$.

 Let $x \in M, r>0, y \in B(x,r/2), w \notin B(x,2r)$ and $ 1 \le k \le  \lfloor \gamma r^\beta \rfloor +1$.
There exists $c_1>0$ such that
\begin{eqnarray}
 \nonumber \EE^{(0,y)} \left[ \ind_{(k,w)} (V_{\tau(x,r)}, X_{\tau(x,r)} )\right] &=& \EE^{(0,y)} \left[ \EE^{(0,y)} \left[ \ind_{(k,w)} (V_{\tau(x,r)} , X_{\tau(x,r)})\vline \F_{k-1} \right] \right] \\
\label{e-t3-1} &=& \EE^{(0,y)} \left[ \ind_{\tau(x,r) > k-1} p(X_{k-1},w)\right]\\
\nonumber &\ge & \PP^{(0,y)} (\tau(x,r) > k-1) \inf_{z \in B(x,r)} p(z,w) \\
\nonumber & \ge &   \PP^{(0,y)} (\tau(x,r) = \lfloor \gamma r^\beta \rfloor+1 ) \inf_{z \in B(x,r)} p(z,w) \\
& \ge & c_1 \frac{1}{V_h(d(x,w))d(x,w)^\beta} \label{e-t3-2}
\end{eqnarray}
for all $x \in M$, for all $r >0$, for all $y \in B(x,r/2)$, for all $w \notin B(x,2r)$ and for all $1 \le k \le \lfloor \gamma r^\beta \rfloor +1$.
The last line follows from \eqref{e-exitX}, \ref{ljp}, \eqref{e-count}, \eqref{e-vc} and the triangle inequality $3 d(x,w) \ge d(z,w)$ for all $z \in B(x,r)$ and for all $w \notin B(x,2r)$.

By \eqref{e-t3-1},\ref{ujp}, \eqref{e-vc} and the triangle inequality $d(z,w) \ge d(x,w)/2$ for all $z \in B(x,r)$, there exists $C_1>0$ such that
\begin{equation}
 \label{e-t3-3} \EE^{(0,x)} \left[ \ind_{(k,w)} (V_{\tau(x,r)}, X_{\tau(x,r)})\right] \le \sup_{z \in B(x,r)} p(z,w) \le  \frac{C_1}{V_h(d(x,w))d(x,w)^\beta}
\end{equation}
for all $x \in M$,  for all $r>0$, for all $w \notin B(x,2r)$ and for all $k \in \N$.

By \eqref{e-t3-2} and \eqref{e-t3-3}, the choice $\theta_3=C_1/c_1$ satisfies \eqref{e-t3}.
\end{proof}

We need the following exit time definition:
\[
 \tau(k,x,r)= \min \{ n \in \N: (V_n,X_n) \notin Q(k,x,r) \}.
\]
As before, we abbreviate $\tau(0,x,r)$ by $\tau(x,r)$.
We are now ready to prove the following parabolic Harnack inequality.
\begin{theorem}[Parabolic Harnack inequality] \label{t-phi}
Under the assumptions of Theorem \ref{t-ubq}, there exist $C_H,R_0 >0$ such that if $q$ is bounded, non-negative on $\T$ and parabolic on $\{0,1,\ldots, \lfloor 8 \gamma R^\beta \rfloor \} \times M$, then
 \begin{equation} \label{e-phi}
  \max_{(k,y) \in Q( \lfloor \gamma R^\beta \rfloor , z ,R/3)} q(k,y) \le C_H \min_{w \in B(z,R/3)} q(0,w)
 \end{equation}
for all $R \ge R_0$, for all $q$ and for all $z \in M$.
\end{theorem}
\begin{proof}
Let $a>0$ be such that $d(x,y) \notin (0,a)$ for all $x,y \in M$. Since $(M,d)$ is uniformly discrete such a constant exists. Choose $R_0 \ge \max(3a, 1)$ such that
 \begin{equation} \label{e-phi1}
  \lfloor \gamma R^\beta \rfloor \ge \lfloor \gamma (R/3)^\beta \rfloor +1
 \end{equation}
for all $R \ge R_0$.

Using Lemma \ref{l-t1} and \eqref{e-phi1}, there exists $c_1\in(0,1)$ such that
\begin{equation} \label{e-phi2}
 \PP^{(k,x)} (T_C <\tau(k,x,r)) \ge c_1
\end{equation}
 for all $r\ge R_0$, for all $(k,x) \in \T$, for all $C \subseteq Q(k+1,x,r/3)$ such that  \[\abs{C}/ \abs{Q(k+1,x,r/3)} \ge 1/3.\]

 By multiplying $q$ by a constant, we may assume that
\[
 \min_{w \in B(z,R/3)} q(0,w) =  q(0,v) = 1
\]
for some $v \in B(z,R/3)$.

 Let $\theta_1,\theta_2,\theta_3$ be the constants from Lemmas \ref{l-t1}, \ref{l-t2} and \ref{l-t3} respectively. Define the constants
 \begin{equation} \label{e-phi3}
\eta:=\frac{c_1}{3}, \hspace{1cm} \zeta:= \frac{c_1}{3} \wedge \frac{\eta}{\theta_3}, \hspace{1 cm} \lambda:= \frac{R_0}{a}.
\end{equation}
Let $\alpha>0$, be a constant satisfying \eqref{e-vc}. By \eqref{e-vc} and \eqref{e-hom}, there exists $C_1>0$ large enough such that, for any $r,R,K>0$ that satisfies
\begin{equation}\label{e-phi4}
\frac{r}{R}=C_1 K^{-1/(\alpha+\beta)} < 1,
\end{equation}
we have
\begin{eqnarray}
\label{e-phi5} \frac{\abs{Q(0,x,r/3)}}{V_h(3^{1/\beta} R) R^{\beta}} &>& \frac{9 }{\theta_1 \zeta K},  \\
\label{e-phi6} \frac{V_h(r/\lambda) (r/\lambda)^{\beta}}{V_h(2^{1+1/\beta}R) (2^{1+1/\beta}R)^{\beta}}  &>& \frac{1}{ \theta_2 \zeta K}.
\end{eqnarray}

We now iteratively choose points $(k_i,x_i) \in Q(\lfloor \gamma R^\beta \rfloor , z, 2R/3)$ for $i=1,2,\ldots$ follows:
The sequence $(k_i,x_i)_{i \in \N^*}$  is chosen such that $K_i=q(k_i,x_i)$ is strictly increasing, that is $K_i < K_{i+1}$ for all $i \in \N^*$.
The starting point  $(k_1,x_1) \in Q(\lfloor \gamma R^\beta \rfloor , z, R/3)$ is chosen such that
\[
 K_1= q(k_1,x_1)= \max_{(k,y) \in Q( \lfloor \gamma R^\beta \rfloor , z ,R/3)} q(k,y).
\]
If $C_1 K_1^{-1/(\alpha+\beta)} \ge 1/3$, then we have \eqref{e-phi}.

Consider the case: $C_1 K_1^{-1/(\alpha+\beta)} < 1/3$.
We now describe a procedure to obtain $(k_{i+1},x_{i+1}) \in Q(\lfloor \gamma R^\beta \rfloor , z, 2R/3)$ given $(k_i,x_i) \in Q(\lfloor \gamma R^\beta \rfloor , z, 2R/3)$ for $i=1,2,3,\ldots$.

Let $r_i$ be defined by
\begin{equation}\label{e-phi0}
 \frac{r_i}{R}= C_1 K_i^{-1/(\alpha+\beta)}.
\end{equation}
Assume that $q \ge \zeta K_i$ on $U_i:= \{ k_i \} \times B(x_i,r_i/\lambda)$.
Since $R \ge R_0$, $\lambda \ge 3$ and $r_i \le R$, we have $(k_i,x_i) \in Q(0,v,2^{1/\beta}R)$, $r_i/\lambda \le 2^{1 + (1/\beta)} R$ and $k \ge 1 + \lfloor \gamma (r_i/3)^\beta \rfloor $.
Therefore by Lemma \ref{l-t2},
\begin{eqnarray*}
1 &=& q(0,v) = \EE^{(0,v)} q\left(V_{T_{U_i} \wedge \tau(0,v,2^{1+(1/\beta)}R)},X_{T_{U_i} \wedge \tau(0,v, 2^{1+(1/\beta)}R) }\right)\\
& \ge & \zeta K \PP^{(0,v)} ( T_U < \tau_{Q(0,v,2^{1+1/\beta}R)})  \ge \frac{\theta_2 \zeta K V_h(r/\lambda) (r/\lambda)^{\beta} }{V_h( 2^{1+1/\beta}R)( 2^{1+1/\beta}R) ^ {\beta}}
\end{eqnarray*}
a contradiction to \eqref{e-phi6}. Therefore there exists $y_i \in B(x_i,r/\lambda)$ such that $q(k_i,y_i)< \zeta K$. Since $\zeta <1$, we have that $y_i \neq x_i$.
Since $(M,d,\mu)$ is uniformly discrete $x_i \neq y_i \in B(x_i,r_i/\lambda)$, we have $r_i/\lambda = (a r_i )/R_0 \ge a$. Hence
\begin{equation}
 \label{e-phi7} r_i \ge R_0.
\end{equation}

If $\EE^{(k_i,x_i)} \left[ q(V_{\tau(k_i,x_i,r_i)},X_{\tau(k_i,x_i,r_i)}); X_{\tau(k_i,x_i,r_i)} \notin B(x,2r) \right] \ge \eta K_i$, we get
\begin{eqnarray*}
\zeta K_i > q(k_i,y_i)&\ge& \mathbb{E}^{(k_i,y_i)} \left[ q(V_{\tau(k_i,x_i,r_i)},X_{\tau(k_i,x_i,r_i)}); X_{\tau_{k,r}} \notin B(x_i,2r_i) \right]\\
& \ge & \theta_3^{-1} \mathbb{E}^{(k_i,x_i)} \left[  q(V_{\tau(k_i,x_i,r_i)},X_{\tau(k_i,x_i,r_i)}); X_{\tau(k_i,x_i,r_i)} \notin B(x_i,2r_i) \right] \\
&\ge& \theta_3^{-1} \eta K_i \ge \zeta K_i
\end{eqnarray*}
a contradiction. In the second line above, we used Lemma \ref{l-t3} and the last line follows from the definition of $\zeta$ in \eqref{e-phi3}. Therefore
\begin{equation}\label{e-phi8}
\EE^{(k_i,x_i)} \left[ q(V_{\tau(k_i,x_i,r_i)},X_{\tau(k_i,x_i,r_i)}); X_{\tau(k_i,x_i,r_i)} \notin B(x,2r) \right] < \eta K_i.
\end{equation}

Define the set
\[
A_i := \{ (j,y) \in Q(k_i+1,x_i,r_i/3) :q(j,y) \ge \zeta K_i \}.
\]
Note that $(k_i,x_i) \in Q(\lfloor \gamma R^\beta \rfloor , z, 2R/3)$, $r_i \le R/3$, $R \ge R_0$, \eqref{e-phi1}, $v \in B(z,R/3)$ and
triangle inequality implies $Q(k_i+1,x_i,r_i/3) \subseteq Q(0,v,3^{1/\beta}R)$.
Therefore by Lemma \ref{l-t1}, we have
\begin{eqnarray*}
1=q(0,v) & \ge & \EE^{(0,v)} \left[ q(V_{T_{A_i} }, X_{T_{A_i}}); T_{A_i} < \tau(v, 3^{1/\beta}R)  \right] \\
& \ge & \zeta K_i \mathbb{P}^{(0,v)} (T_{A_i} <  \tau(v, 3^{1/\beta}R) ) \ge \frac{ \zeta K_i \theta_1 \abs{A_i} }{3V_h(3^{1/\beta} R) R^{\beta}}.
\end{eqnarray*}
This along with \eqref{e-phi5} yields
\[
 \frac{\abs{A_i}}{\abs{Q(k_i+1,x_i,r_i/3)}} \le \frac{3V_h(3^{1/\beta} R) R^{\beta}}{\zeta K_i \theta_1\abs{Q(k_i+1,x_i,r_i/3)}} \le \frac{1}{3}.
\]
Define $C_i= Q(k_i+1,x_i,r_i/3) \setminus A_i$ and $M_i = \max_{Q(k_i+1,x_i,2r_i)}q$.
We write $q(k_i,x_i)$ as
\begin{eqnarray*}
\lefteqn{K_i=q(k_i,x_i)=}\\
&& \EE^{(k_i,x_i)} \left[ q(V_{T_{C_i}},X_{T_{C_i}}); T_{C_i} < \tau(k_i,x_i,r_i) \right] \\
&&+   \EE^{(k_i,x_i)} \left[q(V_{\tau(k_i,x_i,r_i)},X_{\tau(k_i,x_i,r_i)});  \tau(k_i,x_i,r_i) <T_{C_i} , X_{\tau(k_i,x_i,r_i)} \notin B(x_i,2r_i) \right] \\
&&+ \EE^{(k_i,x_i)} \left[q(V_{\tau(k_i,x_i,r_i)},X_{\tau(k_i,x_i,r_i)});  \tau(k_i,x_i,r_i) <T_{C_i} , X_{\tau(k_i,x_i,r_i)} \in B(x_i,2r_i) \right].
\end{eqnarray*}
We use the bound \eqref{e-phi8} for the second term above, to obtain
\begin{equation} \label{e-phi9}
 K_i \le \zeta K_i + \eta K_i + M_i \left( 1 - \PP^{(k_i,x_i)} \left( T_{C_i} < \tau(k_i,x_i,r_i) \right)  \right).
\end{equation}
Combining $\abs{C_i}/\abs{Q(k_i+1,x_i,r_i/3)} \ge 1/3$, \eqref{e-phi7} and \eqref{e-phi2}, we have
\begin{equation}
 \label{e-phi10} \PP^{(k_i,x_i)} \left( T_{C_i} < \tau(k_i,x_i,r_i) \right) \ge c_1.
\end{equation}
By \eqref{e-phi9},\eqref{e-phi10}, \eqref{e-phi3}, we get
\[
 K_i \le \frac{c_1}{3} K_i +\frac{c_1}{3} K_i +(1-c_1) M_i.
\]
It follows that
\begin{equation}\label{e-phi11}
 \frac{M_i}{K_i} \ge 1+ \rho
\end{equation}
where $\rho= c_1/ (3(1-c_1)) >0$.

The point $(k_{i+1},x_{i+1}) \in Q(k_i+1,x_i,2r_i)$ is chosen such that
\[
K_{i+1}= q(k_{i+1},x_{i+1})= M_i = \max_{(j,w) \in Q(k_i+1,x_i,2r_i)} q(j,w).
\]
This along with \eqref{e-phi0} and \eqref{e-phi11} gives
\begin{equation}\label{e-phi12}
 K_{i+1} \ge K_i (1+\rho), \hspace{1cm} r_{i+1} \le r_i (1+\rho)^{-1/(\alpha+\beta)}
\end{equation}
for all $i \in \N^*$.
We will now verify that $(k_{i+1},x_{i+1}) \in Q(\lfloor \gamma R^\beta \rfloor, z, 2R/3)$ for all $i \in \N^*$, if  $K_1$ is sufficiently large.
Using $(k_{i+1},x_{i+1}) \in Q(k_i+1,x_i,2r_i)$, $R_0 \ge 1$ and \eqref{e-phi7}, we have
\begin{eqnarray*}
|k_{i+1} - k_i| &\le& 1 + (2r_i)^{\beta} \le  \left( \frac{2^\beta R_0^\beta +1}{R_0^\beta} \right) r_i^\beta \le 5 r_i^\beta\\
d(x_{i+1},x_i) &\le& 2r_i
\end{eqnarray*}
for all $i \in \N^*$.

Therefore by \eqref{e-phi12} and $(k_1,x_1) \in Q(\lfloor \gamma R^\beta \rfloor, z, R/3)$, we have
\begin{eqnarray}
\label{e-phi13} k_i &\le& \lfloor \gamma R^\beta \rfloor + \gamma (R/3)^\beta +  \frac{5r_1^\beta}{1 - \kappa_1^\beta}, \\
\label{e-phi14} d(x_i,z) &\le& \frac{R}{3} + \frac{2r_1}{1- \kappa_1}
\end{eqnarray}
for all $i \in \N^*$, where $\kappa_1 = (1+\rho)^{-1/(\alpha+\beta)} \in (0,1)$.
Set
\[
 c_2 = \min \left( \frac{1- \kappa_1}{3} , \left( \frac{(1-\kappa_1^\beta)(2^\beta -1)}{5}\right)^{1/\beta} \frac{1}{3}\right).
\]
If $r_1 \le c_2 R$, then by \eqref{e-phi13} and \eqref{e-phi14}, we have that $(k_{i+1},x_{i+1}) \in Q(\lfloor \gamma R^\beta \rfloor, z, 2R/3)$ for all $i \in \N^*$.

If $K_1 \ge (C_1/c_2)^{(\alpha+\beta)}$ by \eqref{e-phi0}, we have $r_1 \le c_2 R$ and therefore $(k_{i+1},x_{i+1}) \in Q(\lfloor \gamma R^\beta \rfloor, z, 2R/3)$ for all $i \in \N^*$.
However \eqref{e-phi7} and \eqref{e-phi12} holds for all $i \in \N^*$, which is a contradiction. Therefore
\[
  \max_{(k,y) \in Q( \lfloor \gamma R^\beta \rfloor , z ,R/3)} q(k,y) = K_1 < (C_1/c_2)^{(\alpha+\beta)}
\]
Therefore \eqref{e-phi} holds with $C_H= (C_1/c_2)^{(\alpha+\beta)}$.
\end{proof}

\section{Heat kernel estimates}\label{s-hke}
In this section, we prove the heat kernel estimates $HKP(\beta)$ for $\beta \in (0,2)$ using the parabolic Harnack inequality \eqref{e-phi}.
We start with the proof of \ref{uhkp}.
\begin{theorem} \label{t-uhkp}
 Under the assumptions of Theorem \ref{t-ubq}, $h_n$ satisfies \ref{uhkp}.
\end{theorem}
\begin{proof}
 By Proposition \ref{p-hitx},
 \begin{eqnarray}
  \nonumber \sum_{z \in B(y,k^{1/\beta})} h_k(x,z) \mu_z &=& \PP^{x}(X_k \in B(y,k^{1/\beta}) \le \PP^{x} (T_X(y,k^{1/\beta}) \le k) \\
\label{e-uh1}  & \le & C_1 \frac{k V_h(k^{1/\beta})}{V_h(d(x,y))(1+d(x,y))^{\beta}}
 \end{eqnarray}
for all $k \in \N^*$ and for all $x,y \in M$. By \eqref{e-uh1} and \eqref{e-hom}, there exists $C_2>0$ such that
\begin{equation} \label{e-uh2}
\min_{z \in B(y,k^{1/\beta})} h_k(x,z) \le  C_2 \frac{k }{V_h(d(x,y))(1+d(x,y))^{\beta}}
\end{equation}
for all $x,y \in M$ and for all $k \in \N^*$.

 Let $R>0$ be defined to satisfy $\gamma R^\beta= n$.
 Since we can take $\gamma < 3^{-2} \le 3^{-\beta}$ without loss of generality, we have $R/3 \ge n^{1/\beta}$.

 By Lemma \ref{l-parab},  $q(k,w)=h_{8n-k}(x,w)$ is parabolic on
 $\{0,1,\ldots, \lfloor 8 \gamma R^\beta \rfloor \} \times M$. By \eqref{e-uh2} and $R/3 \ge n^{1/\beta}$, we have
 \begin{equation}\label{e-uh3}
  \min_{z \in B(y,R/3)} q(0,z) \le \min_{z \in B(y,n^{1/\beta})} h_{8n}(x,z) \le  C_2 \frac{8n }{V_h(d(x,y))(1+d(x,y))^{\beta}}.
 \end{equation}
Since
\begin{equation*}
 h_{7n}(x,y) \le \max_{(k,w) \in Q(\lfloor \gamma R^\beta \rfloor, y, R/3)} q(k,w),
\end{equation*}
by \eqref{e-uh3} and parabolic Harnack inequality \eqref{e-phi}, there exist $C_3,N_0>0$ such that
\begin{equation}\label{e-uh4}
 h_{7n} (x,y) \le C_3 \frac{n }{V_h(d(x,y))(1+d(x,y))^{\beta}}
\end{equation}
for all $x,y \in M$ and for all $n \in \N$ with $n \ge N_0$. Combining \eqref{e-1s}, \eqref{e-uh4} along with Theorem \ref{t-dub} yields \ref{uhkp}.
\end{proof}
\begin{rem} We sketch an alternate proof of Theorem \ref{t-uhkp} that doesn't require parabolic Harnack inequality.
 Using the comparison techniques between discrete time and corresponding continuous time Markov chains
 developed by T. Delmotte(\cite{Del}), we can prove \ref{uhkp} for the kernel $h_n$ using the upper bound for $q_t$ given in Theorem \ref{t-ubq}.
 (see Lemma 3.5 and Theorem 3.6 of \cite{Del})
\end{rem}
We now obtain a near diagonal lower estimate for $h_n$ using parabolic Harnack inequality.
\begin{lem}\label{l-ndlb}
Under the assumptions of Theorem \ref{t-ubq}, there exist $c_1,c_2>0$ such that
\begin{equation}\label{e-ndlb}
 h_n(x,y) \ge \frac{c_1}{V_h(n^{1/\beta})}
\end{equation}
for all $n \in \N^*$ and for all $x,y \in M$ such that $d(x,y) \le c_2 n^{1/\beta}$.
\end{lem}
\begin{proof}
 By Proposition \ref{p-exitX}, there exists $C_1>0$ such that
 \begin{equation*}
\PP^{x}\left(X_n \notin B(x,C_1 n^{1/\beta})\right) \le \PP^{x}( \max_{k \le n}d(X_k,x) > C_1 n^{1/\beta}) \le 1/2
 \end{equation*}
for all $x \in M$ and for all $n \in \N^*$.
Thus
\begin{equation}\label{e-ndl1}
 \sum_{y \in B(x,C_1 n^{1/\beta})} h_n(x,y) \mu_y = 1- \PP^{x}\left(X_n \notin B(x,C_1 n^{1/\beta})\right) \ge \frac{1}{2}
\end{equation}
for all $x \in M$ and for all $n \in \N^*$. Hence, there exists $c_3>0$ such that
\begin{eqnarray}
 h_{2n}(x,x) &= &\sum_{y \in M} h_n^2(x,y) \mu_y \ge \sum_{y \in  B(x,C_1 n^{1/\beta})} h_n^2(x,y) \mu_y \nonumber \\
 & \ge & \frac{1}{V(x,C_1n^{1/\beta})} \left( \sum_{y \in B(x,C_1 n^{1/\beta})} h_n(x,y) \mu_y  \right)^2 \ge \frac{1}{4 V(x,C_1 n^{1/\beta})} \nonumber \\
 & \ge & \frac{c_3}{V_h(n^{1/\beta})} \label{e-ndl2}
\end{eqnarray}
for all $x \in M$ and for all $n \in \N^*$. The second line above follows from Cauchy-Schwarz and \eqref{e-ndl1} and last line follows from \eqref{e-hom} and \eqref{e-vc}.
Combining \eqref{e-ndl2}, \eqref{e-1s}, \eqref{e-vc} and \ref{ljp}, there exists $c_4>0$ such that
\begin{equation}\label{e-ndl3}
 h_n(x,x) \ge  \frac{c_4}{V_h(n^{1/\beta})}
\end{equation}
for all $n \in \N^*$ and for all $x \in M$.

Let $n \in \N^*$ and let $R$ be defined by $n= \gamma R^\beta$.
As in the proof of Theorem \ref{t-uhkp}, we define the function $q(k,w)= h_{8n-k}(x,w)$ which is parabolic on
$\{0,1,\ldots, \lfloor 8 \gamma R^\beta \rfloor\}\times B(x,r)$.
By \eqref{e-phi}, \eqref{e-ndl3} and \eqref{e-vc},  there exist $c_5,C_2, N_0>0$ such that
\begin{eqnarray}
 \min_{z \in B(y,C_2 n^{1/\beta})} h_{8n}(x,z) &\ge&   \min_{z \in B(y,R/3)} q(0,z) \ge  C_H^{-1} h_{7n}(x,x)  \nonumber \\
\label{e-ndl4} &\ge& C_H^{-1} \frac{c_4}{V_h((7n)^{1/\beta})} \ge \frac{c_5}{V_h(n^{1/\beta})}
\end{eqnarray}
for all $n \in \N^*$ with $n \ge N_0$ and for all $x \in M$. Combining \eqref{e-1s}, \eqref{e-ndl4} and \ref{ljp}, we get the desired near diagonal lower bound \eqref{e-ndlb}.
\end{proof}
Next, we prove the full lower bound \ref{lhkp}. This can be done using parabolic Harnack inequality as in Theorem 5.2 of \cite{BL}.
However, we prove the off-diagonal lower bound using a probabilistic argument which relies on the exit time estimate of Proposition \ref{p-exitX}.
\begin{theorem}\label{t-lhkp}
Under the assumptions of Theorem \ref{t-ubq}, $h_n$ satisfies \ref{lhkp}.
\end{theorem}
\begin{proof}
  By Proposition \ref{p-exitX}, there exists $C_1>0$ such that
 \begin{equation} \label{e-lb1}
\PP^{x}\left(X_n \notin B(x,C_1 n^{1/\beta})\right) \le \PP^{x}( \max_{k \le n}d(X_k,x) > C_1 n^{1/\beta}) \le 1/2
 \end{equation}
for all $x \in M$ and for all $n \in \N^*$.

We will first handle the case $d(x,y) \ge 3 C_1  n^{1/\beta}$. Define the event
\[
A_k :=  \{ X_0=x, X_n=y, \tau_{X}(x,C_1 n^{1/\beta})=k ,X_j \in B(y,C_1 n^{1/\beta}) \forall j \in [k,n] \cap \N\}
\]
for $k=1,2,\ldots,n$.
By \eqref{e-count}, we have
\begin{equation}\label{e-lb2}
 h_n(x,y) = p_n(x,y) \mu_y^{-1} \ge C_\mu^{-1} p_n(x,y) \ge C_\mu^{-1} \sum_{k=1}^n \PP(A_k).
\end{equation}
By reversibility \eqref{e-reverse} and \eqref{e-count},
\begin{eqnarray}
\nonumber \lefteqn{\PP(X_0=z_0,X_1=z_1, \ldots X_r=z_r)} \\
\nonumber &=& \frac{\mu_{z_r}}{\mu_{z_0}} \PP(X_0=z_r,X_1=z_{r-1}, \ldots X_r=z_0)  \\
& \ge & C_\mu^{-2} \PP(X_0=z_r,X_1=z_{r-1}, \ldots X_r=z_0) \label{e-lb3}
 \end{eqnarray}
for all $r \in \N$  and for all $z_i$'s in $M$. By \eqref{e-lb3}, there exists $c_3>0$ such that
\begin{eqnarray}
 \PP(A_k) &\ge& C_\mu^{-2} \sum_{x_{k-1}, x_k \in M } \left\{ \PP^{x}( \tau_X(x,C_1n^{1/\beta}) > k-1, X_{k-1}=x_{k-1} ) \right.  \nonumber\\
& &  \times \left. p(x_{k-1},x_k) \times \PP^{y}( \tau_X(y,C_1n^{1/\beta}) > n- k, X_{n-k}=x_k) \right\}  \nonumber \\
&\ge& \frac{c_3}{V_h(d(x,y)) (1+d(x,y))^{1/\beta}} \label{e-lb4}
\end{eqnarray}
for all $n \in \N^*$ , for all $1 \le k \le n$ and for all $x,y \in M$ with $d(x,y) \ge 3 C_1 n^{1/\beta}$.
The last line follows from triangle inequality $d(x_{k-1},x_k) \ge d(x,y) - 2C_1 n^{1/\beta} \ge d(x,y)/3$, along with \ref{ljp}, \eqref{e-vc}, \eqref{e-count} and \eqref{e-lb1}.
Combining \eqref{e-lb2} and \eqref{e-lb4}, there exists $c_4>0$ such that
\begin{equation}\label{e-lb5}
 h_n(x,y) \ge \frac{c_4 n}{V_h(d(x,y)) (1+d(x,y))^\beta}
\end{equation}
for all $n \in \N^*$ and for all $x,y \in M$ with $d(x,y) \ge 3C_1 n^{1/\beta}$.
Thus we get the desired lower bound for case $d(x,y) \ge 3C_1 n^{1/\beta}$.

Let $c_2 >0$ be the constant from Lemma \ref{l-ndlb}. By Lemma \ref{l-ndlb}, it remains to check the case $c_2 n^{1/\beta} < d(x,y) < 3C_1 n^{1/\beta}$. Choose $K \in \N$, so that
\begin{equation*}
 K \ge \left( \frac{6C_1}{c_2}\right)^\beta \vee (1 - 2^{-\beta})^{-1}.
\end{equation*}
Then if $n \in \N^*$ with  $d(x,z) \ge c_2 n^{1/\beta}/2 $, we have that
\[
 d(x,z) \ge 3 C_1 \left\lfloor \frac{n}{K} \right\rfloor ^{1/\beta},   \left( n -\left\lfloor \frac{n}{K} \right\rfloor  \right)^{1/\beta} \ge \frac{n^{1/\beta}}{2}.
\]
If $d(x,y) \ge c_2 n^{1/\beta}$ and $z \in B(y, c_2 n^{1/\beta}/2)$, then $d(x,z) \ge c_2 n^{1/\beta}/2 \ge d(y,z)$.
Therefore by the above inequalities, we have
\[
 d(x,z) \ge 3 C_1 \left\lfloor \frac{n}{K} \right\rfloor ^{1/\beta},  d(z,y) \le c_2 \left( n -\left\lfloor \frac{n}{K} \right\rfloor  \right)^{1/\beta}
\]
for all $z \in B(y,c_2n^{1/\beta}/2)$ and for all $x$ such that $d(x,y) > c_2 n^{/\beta}$.
By Chapman-Kolmogorov equation \eqref{e-ck}
\begin{eqnarray}
 h_n(x,y) \ge \sum_{z \in B(y, c_2 n^{1/\beta}/2)} h_{\left\lfloor \frac{n}{K} \right\rfloor } (x,z) h_{n - \left\lfloor \frac{n}{K} \right\rfloor} (z,y) \mu_z
\end{eqnarray}
We use \eqref{e-lb5} to estimate  $h_{\left\lfloor \frac{n}{K} \right\rfloor } (x,z)$ and Lemma \ref{l-ndlb} to estimate $h_{n - \left\lfloor \frac{n}{K} \right\rfloor} (z,y)$ and use \eqref{e-vc}, to get a constant $c_5>0$ such that
\begin{eqnarray}
 h_n(x,y) \ge \frac{c_5}{V_h(n^{1/\beta})}
\end{eqnarray}
for all $n \in \N$ with $n \ge K$ and for all $x,y \in M$ with $c_2 n^{1/\beta} < d(x,y) < 3C_1 n^{1/\beta}$.
The case $n < K$ follows from \eqref{e-1s} along with \ref{ljp}.
\end{proof}
\section{Generalization to Regularly varying functions }\label{s-gen}
In this section, we replace $(1+d(x,y))^\beta$ in \ref{ljp} and \ref{ujp} by a general regularly varying function $\phi$ of index $\beta$. For a comprehensive introduction to regular variation, we refer the reader to  \cite{BGT}.
\begin{definition}
 Let $\rho \in \R$. We say $\phi:[0,\infty) \to \R$ is \emph{regularly varying of index $\rho$}, if $\lim_{x \to \infty} \phi(\lambda x)/ \phi(x) \to \lambda^\rho$ for all $\lambda >0$.
 A function $l:[0,\infty) \to \R$ is \emph{slowly varying} if $l$ is regularly varying of index $0$.
\end{definition}
We generalize Theorems \ref{t-uhkp} and \ref{t-lhkp} for more general jump kernels using a change of metric argument. We change the metric $d$ by composing it with an appropriate concave function, so that under the changed metric
the jump kernel satisfies \ref{ljp} and \ref{ujp}. The following Lemma provides us with the concave function we need.
\begin{lemma} \label{l-concave}
 Let $\rho \in (0,1)$ and $f:[0,\infty) \to (0,\infty)$ be a continuous, positive regularly varying function with index $\rho$. Then there exists a concave, strictly increasing function $g$ and a constant $C >0$ such that
 $g(0)=0$, $\lim_{x \to \infty} g(x)/f(x)=1$ and
 \[
 C^{-1} \le \frac{ 1+g(x)}{f(x)} \le C
 \]
for all $x \in [0,\infty)$.
\end{lemma}
\begin{proof}
 By Theorem 1.8.2 and Proposition 1.5.1 of \cite{BGT}, there exists $A>0$ and a smooth function $f_1:(A,\infty) \to (0,\infty)$ such that
 \[
\lim_{x \to \infty}  x^nf_1^{(n)}(x)/f_1(x) \to \rho(\rho -1)\ldots (\rho -n +1), \hspace{1cm} \forall n = 1,2,\ldots,
 \]
 $\lim_{x\to \infty} f_1(x)/f(x)=1$.

 Let $B = (A+2) f_1^{(1)} (A+1)$. Then the function
 \begin{equation}
  g(x) = \begin{cases} x (B+f_1(A+1))/(A+1) &\mbox{if } x \le A+1 \\
 B+ f_1(x) & \mbox{if } x \ge A+1 \end{cases}
 \end{equation}
is concave. It is clear the $g$ is strictly increasing, continuous, $g(0)=0$ and $\lim_{x \to \infty} g(x)/f(x)=1$.
The constant $C>0$ can be obtained from the fact the $f$ is continuous and positive and $g$ is continuous and non-negative.
\end{proof}
\begin{proof}[Proof of Theorem \ref{t-main}]
 Choose $\delta \in (\beta,2)$. Then $\phi^{1/\delta}$ is regularly varying with index $\beta/\delta \in (0,1)$. By Lemma \ref{l-concave}, there exists a
 continuous, concave, strictly increasing function $g$ and a constant $C_3>0$ such that $g(0)=0$, \[ \lim_{x\to \infty} g(x)/f^{1/\delta}(x)=1 \] and $C_3^{-1} \le (1+g(x))^\delta/f(x) \le C_3$ for all $x \ge 0$.
 Define the new metric $d'(x,y) = g(d(x,y))$. Since $g$ is strictly increasing, $d'$ is uniformly discrete metric. By $V'(x,r)$, we denote the volume of balls of radius $r$ for the metric measure space $(M,d',\mu)$, that is
 \[
  V'(x,r)= \mu \left( \{ y \in M : d'(x,y) \le r \} \right).
 \]
By \eqref{e-hom}, there exists $C_4>0$
\[
C_4^{-1} V_h'(r)\le  V'(x,r) \le C_4 V_h'(r)
\]
for all $x \in M$ and for all $r>0$ where $V_h'(r):= V_h(g^{-1}(r))$ and $g^{-1}:[0,\infty) \to [0,\infty)$ denotes the inverse of $g$.
By \cite[Proposition 1.5.15]{BGT} and \eqref{e-vc},  we have that there exists $C_5>0$ such that
\begin{equation} \label{e-mn1}
C_5^{-1} V_h( r^{\delta/\beta} l_\# (r^{\delta/\beta}))\le  V_h'(r) \le C_5 V_h( r^{\delta/\beta} l_\# (r^{\delta/\beta}))
\end{equation}
for all $r>0$.
By the properties of $g$ and \eqref{e-mn}, there exists $C_6>0$
 \begin{equation*}
 C_6^{-1}\frac{1}{V_h'(d'(x,y)) (1+d'(x,y))^\delta } \le  J(x,y)=J(y,x) \le C_6 \frac{1}{V_h'(d'(x,y)) (1+d'(x,y))^\delta}
 \end{equation*}
 for all $x,y \in M$. Therefore by Theorem \ref{t-uhkp}, there exists $C_7>0$ such that
\[
 h_n(x,y) \le C_7 \left( \frac{1}{V_h'( n^{1/\delta})} \wedge \frac{n}{V_h'(d'(x,y)) (1+d'(x,y))^\delta}\right).
\]
for all $n \in \N^*$ and for all $x,y \in M$.
Combining with \eqref{e-mn1} and \eqref{e-vc}, there exists $C_8>0$ such that
\[
 h_n(x,y) \le C_8 \left( \frac{1}{ V_h(n^{1/\beta}l_{\#}(n^{1/\beta}))} \wedge \frac{n}{ V_h(d(x,y))\phi(d(x,y))} \right)
\]
for all $n \in \N^*$ and for all $x,y \in M$. A similar argument using Theorem \ref{t-lhkp} gives the desired lower bound on $h_n$.
\end{proof}
\subsection*{Acknowledgments}
The authors thank Tianyi Zheng for useful discussions.
\bibliographystyle{amsalpha}

\end{document}